\documentclass[a4paper,12pt]{amsart}
\usepackage[lmargin=3.7cm,rmargin=2.7cm,bmargin=3.3cm,tmargin=3cm]{geometry}

\usepackage[latin1]{inputenc}
\usepackage{amsmath}
\usepackage{amssymb}  % More symbols
\usepackage{amsthm}

\usepackage{wasysym}

\usepackage{textcomp}
\usepackage{enumerate}

\usepackage{url}

\usepackage[sort,numbers]{natbib}
\usepackage{array}

\usepackage{psfrag}
\usepackage{graphicx}
\usepackage{ifthen}
\usepackage{bbm}
\newcommand{\tr}{\mathop{\mathrm{tr}}}
\newcommand{\defeq}{\mathrel{\mathop:}=} 
 
\newcommand{\norm}[2][]{
  \left\| 
  \ifthenelse{\equal{#1}{}}{#2}{\vphantom{#1}\smash{#2}}
  \right\|}

\newcommand{\Econd}[2]{\mathbb{E}\left[\left.#1\;\right|\;#2\right]}

\newcommand{\ud}{\mathrm{d}}
\newcommand{\ball}{\overline{B}}
\newcommand{\R}{\mathbb{R}}
\newcommand{\F}{\mathcal{F}}

\renewcommand{\P}{\mathbb{P}}
\newcommand{\spc}[1]{\mathbb{#1}}
\newcommand{\charfun}[1]{\mathbbm{1}_{#1}}
\newcommand{\uarg}{\mathrel{\cdot}}
\newcommand{\acc}{\mathop{\mathrm{acc}}}

\newtheorem{theorem}{Theorem}
\newtheorem{corollary}[theorem]{Corollary}
\newtheorem{lemma}[theorem]{Lemma}
\newtheorem{proposition}[theorem]{Proposition}

\theoremstyle{definition}
\newtheorem{definition}[theorem]{Definition}
\newtheorem{assumption}[theorem]{Assumption}

\theoremstyle{remark}

\newcounter{saveenum}

\newtheorem{remark}[theorem]{Remark}

\newenvironment{prf-lemma-proposal}{
  \begin{proof}[\textbf{Proof of Lemma \ref{lemma:upper-proposal}}]
      }{\end{proof}}
\newenvironment{prf-prop-lower}{
  \begin{proof}[\textbf{Proof of Proposition \ref{prop:lower}}]
   }{\end{proof}}
\newenvironment{prf-upper-polynomial}{
  \begin{proof}[\textbf{Proof of Proposition
    \ref{prop:upper-polynomial}}]
  }{\end{proof}}
\newenvironment{prf-half-space}{
  \begin{proof}[\textbf{Proof of Lemma \ref{lemma:half-space}}]
      }{\end{proof}}
\newenvironment{prf-compact-stable}{
  \begin{proof}[\textbf{Proof of Theorem
    \ref{thm:compact-stable}}]
  }{\end{proof}}
\newenvironment{prf-super-exp-ergodicity}{
  \begin{proof}[\textbf{Proof of Theorem
    \ref{thm:super-exp-ergodicity}}]
  }{\end{proof}}

\date{\today}

\begin{document}
\sloppy
%%%%%%%%%%%%%%%%%%%%%%%%%%%%%%%%%%%%%%%%%%%%%%%%%%%%%%%%%%%%%%%%%%%%%%%%%%%%%%%
% Document related definitions. %{{{
%%%%%%%%%%%%%%%%%%%%%%%%%%%%%%%%%%%%%%%%%%%%%%%%%%%%%%%%%%%%%%%%%%%%%%%%%%%%%%%
% Title, author, date, acknowledgements goes here:
\title[Stability and ergodicity of adaptive scaling metropolis]{On the stability and
  ergodicity of adaptive scaling Metropolis algorithms}
\author{Matti Vihola}
\address{Matti Vihola, Department of Mathematics and Statistics,
  University of Jyväskylä,
  P.O.Box 35,
  FI-40014 University of Jyväskylä,
  Finland}
\email{matti.vihola@iki.fi}
\urladdr{http://iki.fi/mvihola/}
\thanks{The author was supported 
  by the Academy of Finland, projects no.~110599 and 201392,
  by the Finnish Academy of Science and Letters, Vilho,
  Yrjö, and Kalle Väisälä Foundation,
  by the Finnish Centre of Excellence in Analysis and
  Dynamics Research,
  and by the Finnish Graduate School in
  Stochastics and Statistics.}
%}}}

%%%%%%%%%%%%%%%%%%%%%%%%%%%%%%%%%%%%%%%%%%%%%%%%%%%%%%%%%%%%%%%%%%%%%%%%%%%%%%%
\begin{abstract} %{{{
  The stability and ergodicity properties of two 
  adaptive random walk Metropolis algorithms are considered. 
  The both algorithms adjust the scaling of the
  proposal distribution continuously based on the observed acceptance
  probability. Unlike the previously proposed forms of the algorithms, the
  adapted scaling parameter is not constrained within a predefined compact
  interval. The first algorithm is based on scale adaptation only,
  while the second one incorporates also covariance adaptation.
  A strong law of large numbers is shown to hold 
  assuming that the target density is smooth enough and
  has either compact support or super-exponentially decaying tails.
\end{abstract} %}}}
\maketitle
\thispagestyle{empty}

%%%%%%%%%%%%%%%%%%%%%%%%%%%%%%%%%%%%%%%%%%%%%%%%%%%%%%%%%%%%%%%%%%%%%%%%%%%%%%%
\section{Introduction} 
\label{sec:intro} %{{{

Markov chain Monte Carlo (MCMC) is
a general method often used to approximate integrals of the type
\[
    I \defeq \int_{\R^d} f(x) \pi(x) \ud x < \infty
\]
where $\pi$ is a probability density function \cite[see,
e.g.,][]{robert-casella,gilks-mcmc,roberts-rosenthal-general}. 
The method is based on a 
Markov chain $(X_n)_{n\ge 1}$ that can be simulated in
practice, and for which the ergodic averages 
$I_n \defeq n^{-1}\sum_{k=1}^n f(X_k)$ converge to the integral 
$I$ as $n\to
\infty$. Such a chain can be constructed, for example, as follows.
Assume $q$ is a standard Gaussian probability density in $\R^d$, and
let $X_1 \equiv x_1$ for some fixed point $x_1 \in \R^d$.
For $n\ge 2$, recursively,
\begin{enumerate}[(S1)]
\item set $Y_n \defeq X_{n-1} + \theta \Sigma^{1/2} W_n$, where $W_n$ are independent random
      vectors distributed according to $q$, and
      \label{item:proposal-step}
\item with probability $\alpha_n \defeq \min\{1,\pi(Y_n)/\pi(X_{n-1})\}$ 
      the proposal is accepted and $X_n = Y_n$; otherwise 
      the proposal is rejected 
      and $X_n = X_{n-1}$.
      \setcounter{saveenum}{\value{enumi}}
\end{enumerate}
For any scale $\theta>0$ and symmetric positive definite (covariance) matrix
$\Sigma\in\R^{d\times d}$ 
this symmetric random walk Metropolis algorithm is valid:
$I_n\to I$ almost surely as $n\to\infty$ \cite[e.g.][Theorem 1]{nummelin-mcmcs}.
However, the efficiency of the method, that is, the speed at which
$I_n$ converges to $I$, is crucially affected by the choice of $\theta$ and
$\Sigma$. Suppose for a moment that the matrix $\Sigma$ is fixed, and we
only vary $\theta>0$. Then, for too large $\theta$, 
few proposals become accepted and the chain mixes poorly.
If $\theta$ is too small, most of the proposals $Y_n$ become accepted, but 
the steps $X_n-X_{n-1}$ are small, preventing good mixing.
In fact, previous results indicate that the acceptance probability is
closely related with the efficiency of the algorithm. Commonly used 
`rule of thumb' is that
the acceptance probability $\alpha_n$ should be on the
average about $0.234$
even though this choice is not always optimal
\cite{roberts-rosenthal-scaling,roberts-gelman-gilks-scaling,bedard-spa,sherlock-roberts}.
In practice, such a $\theta$ is usually found by several trial runs,
which can be laborious and time-consuming.

So called adaptive MCMC  algorithms have gained
popularity since the seminal work of Haario, Saksman, and Tamminen
\cite{saksman-am}.  Several other such algorithms have been proposed after
Andrieu and Robert \cite{andrieu-robert} noticed the connection
between Robbins-Monro stochastic approximation and adaptive MCMC
\cite{andrieu-moulines,atchade-rosenthal,roberts-rosenthal,roberts-rosenthal-examples,andrieu-thoms}.
The adaptive scaling Metropolis (ASM) algorithm 
optimises the scaling $\theta>0$ of the proposal
distribution adaptively, based on the observed acceptance probability.
Namely, in the step 
(S\ref{item:proposal-step}) of the above algorithm, the constant
$\theta$ is replaced, for example, with 
$\theta_{n-1}\defeq e^{S_{n-1}}$ where $(S_n)_{n\ge
  1}$ are random variables with $S_1\equiv s_1\in\R$ and
for $n\ge 2$ defined recursively as follows
\begin{enumerate}[(S1)]
    \setcounter{enumi}{\value{saveenum}}
    \item \label{item:adapt-theta} 
      $S_n = S_{n-1} + \eta_n (\alpha_n -
      \alpha^*)$
\end{enumerate}
where the parameter $\alpha^*$ determines 
the desired mean acceptance probability,
often $0.234$, and $(\eta_n)_{n\ge 2}$ is a sequence of
positive adaptation step sizes decaying to zero.

A similar random walk Metropolis algorithm with adaptive scaling was actually 
proposed over a decade
ago by Gilks, Roberts and Sahu \cite{gilks-roberts-sahu}.  Their approach
differed from the ASM approach so that the adaptation was 
performed only at particular regeneration times,
which may occur infrequently or
may be difficult to identify in practice.
The ASM algorithm presented above
has been proposed earlier by several authors 
\cite{andrieu-robert,atchade-rosenthal,%
roberts-rosenthal-examples}, with a slightly different update
formula (S\ref{item:adapt-theta}). The exact form of
(S\ref{item:adapt-theta}) was used by
\cite{atchade-fort,andrieu-thoms}. The crucial difference 
of the present paper compared to the earlier works is 
that the algorithm does not involve any
additional constraints on $\theta_n$.
This difference is chiefly a theoretical advance, as discussed below.
Therefore, no empirical studies of the performance of the algorithms are
included in the paper. 

Since the ASM algorithm only adapts the scale of the
proposal distribution, it is likely to be inefficient in certain situations.
For example, if $\pi$ is high-dimensional and possesses a strong correlation
structure and $\Sigma$ does not match this structure,
the ASM approach is likely to be suboptimal. In such a situation,
one can employ the Adaptive Metropolis (AM) algorithm \cite{saksman-am}
to adapt the covariance shape with the scaling adaptation
\cite{atchade-fort,andrieu-thoms}. That is, in addition to using random
$\theta_{n-1}$ in (S\ref{item:proposal-step}), one uses also a random
matrix $\Sigma_{n-1}$ instead of a fixed $\Sigma$. Namely, $\Sigma_n$ is 
a covariance estimator based on $X_1,\ldots,X_n$; the details can be found in
Section \ref{sec:main}. This algorithm will be referred here 
to as the adaptive scaling within AM (ASWAM).

It is not obvious that adaptive algorithms like the ASM and the ASWAM are
valid, that is, $I_n \to I$. In fact, there are examples of continuously
adapting MCMC schemes that destroy the correct 
ergodic\footnote{In the present work, the word `ergodicity' refers 
  to the convergence of
  ergodic averages $I_n$ to $I$, unlike 
  Roberts and Rosenthal \cite{roberts-rosenthal} who define
  `ergodic' by the convergence of the marginal distributions of $X_n$ to $\pi$ in the
  total variation sense.}
properties 
\cite{roberts-rosenthal}.  Current ergodicity results on adaptive MCMC
algorithms ensuring that $I_n\to I$ assume some `uniform' behaviour for all
the possible MCMC kernels
\cite{atchade-rosenthal,roberts-rosenthal,atchade-fort}. In the context of
the adaptive scaling framework, this essentially means that $\theta_n$ must
be constrained to a predefined set $[a,b]$ with some $0<a\le b<\infty$.
Alternatively, one can use a general reprojection technique with a sequence
of such sets $[a_n,b_n]$ with $a_n\searrow 0$ and $b_n\nearrow \infty$ as
proposed by Andrieu and Moulines \cite{andrieu-moulines}, or stabilisation
methods that modify the adaptation rule to ensure stable behaviour
\cite{andrieu-thoms}. Such constraints and stabilisation structures are
theoretically convenient, but may pose a problem for a practitioner. Good
values for the constraint parameters may be difficult to choose without
prior knowledge of the target distribution $\pi$. In the worst case, the
values are chosen inappropriately and the algorithm is rendered useless in
practice.

It is a common belief that many of the proposed adaptive MCMC
algorithms are inherently stable and thereby do not require additional
constraints or stabilisation structures. Indeed, there is considerable
empirical evidence of the stability of several unconstrained algorithms,
including the adaptive scaling approach. There are yet only few theoretical
results, especially Saksman and Vihola \cite{saksman-vihola} verifying the
correct ergodic properties and the stability of the AM
algorithm \cite{saksman-am}, provided the target distribution $\pi$ has
super-exponentially decaying tails with regular contours. These assumptions
on $\pi$ are close to those that ensure the geometric ergodicity of a
non-adaptive random walk Metropolis algorithm \cite{jarner-hansen}. The
result in \cite{saksman-vihola} does not assume an upper bound, but 
requires an explicit lower bound for the adapted covariance 
parameter.\footnote{The recent work \cite{vihola-collapse} gives partial stability results of the AM
also without the lower bound.}
In the context of the scaling adaptation, the
lower bound is analogous to constraining $\theta_n$ to the interval
$[a,\infty)$, where $a>0$.

The main results of this paper, 
formulated in the next section, show 
that the stability and ergodicity of the ASM algorithm can be verified 
under similar assumptions on the target distribution as in
\cite{saksman-vihola}, without any modifications or
constraints on the adaptation parameter $\theta_n\in(0,\infty)$.
These are the first results that validate the correctness of a 
completely unconstrained, fully adaptive MCMC algorithm.
Similar result applies for the ASWAM approach, given that stability is
enforced on the covariance parameter $\Sigma_n$ by bounding the
eigenvalues away from zero and infinity.

%}}}

%%%%%%%%%%%%%%%%%%%%%%%%%%%%%%%%%%%%%%%%%%%%%%%%%%%%%%%%%%%%%%%%%%%%%%%%%%%%%%%
\section{Main results} 
\label{sec:main} %{{{

The scaling adaptation introduced in Section \ref{sec:intro} 
can be generalised by considering a function $\phi$ 
mapping real-valued parameter values $S_n$ to a scaling in $(0,\infty)$.
%%%%%%%%%%%%%%%%%%%%
\begin{assumption} 
    \label{a:scaling-function} %{{{
    The scaling function $\phi:\R\to(0,\infty)$ is increasing and surjective,
    piecewise differentiable and 
    there are constants $h,c>0$ and $\kappa\ge 1$ such that
    \begin{equation*}
        \phi'(s+\bar{h}) \le c \max\{1,\phi^{\kappa}(s)\}
    \end{equation*}
    for all $s\in\R$ and all $0\le \bar{h} \le h$. 
\end{assumption} %}}}
The function $\phi(s) = e^s$ was suggested above, but
Assumption \ref{a:scaling-function} allows one to use also, for example,
piecewise polynomially defined $\phi$. For example, defining 
$\phi(x)=x$ whenever $x$ is greater than some $x_0>0$ and 
continuing $\phi$ appropriately for $x<x_0$
gives an algorithm in the spirit of
Atchadé and Rosenthal \cite{atchade-rosenthal}.

The results hold also for other than a Gaussian proposal, as long as the
proposal density is spherically symmetric and satisfies a certain tail
behaviour.
%%%%%%%%%%%%%%%%%%%%
\begin{assumption} 
    \label{a:proposal} %{{{
The proposal density $q$ can be written as 
$q(z) = \hat{q}(\| z\|)$ where 
$\hat{q}:[0,\infty)\to(0,\infty)$ is a bounded, decreasing and 
differentiable function. Moreover, for any $\xi\in(0,1)$ 
there exist an $\epsilon^*>0$, constants $0\le a < b < \infty$ and
$c_1,c_2,c_3>0$ 
such that for all $\epsilon\in[0,\epsilon^*]$, the following bounds hold
for the derivative of $\hat{q}$
    \begin{eqnarray*}
        \xi\hat{q}'(x) - \hat{q}'(x+\epsilon) &\ge& c_1,\qquad\text{for all}\qquad
        a\le x\le b, \\
        \int_0^\infty \min\{0,\xi\hat{q}'(x) - \hat{q}'(x+\epsilon)\} \ud x 
        &\ge& - c_2e^{-c_3 \epsilon^{-1}}.
    \end{eqnarray*}
\end{assumption} %}}}
Proposition \ref{prop:proposal-examples} in Appendix \ref{sec:metropolis} shows that
Assumption \ref{a:proposal} holds for Gaussian and Student distributions
$q$.

We also need certain conditions for the adaptation step size sequence
$(\eta_n)_{n\ge 2}$.
%%%%%%%%%%%%%%%%%%%%
\begin{assumption}
    \label{a:step-size} %{{{
    The sequence $(\eta_n)_{n\ge 2}$ is non-negative,
    $\sum_{n=2}^\infty \eta_n = \infty$ and
    $\sum_{n=2}^\infty \eta_n^2 < \infty$.
\end{assumption} %}}}
Assumption \ref{a:step-size} is classical in the context of 
stochastic approximation.
A typical choice for the step size sequence satisfying Assumption
\ref{a:step-size} is 
$\eta_n \propto n^{-\gamma}$ with some constant $\gamma\in(1/2,1]$.

We are now ready to define the adaptive scaling Metropolis (ASM) and the
adaptive scaling within adaptive Metropolis (ASWAM) algorithms.
%%%%%%%%%%%%%%%%%%%%
\begin{definition}[ASM]
    \label{def:asm} %{{{
Suppose that the matrix
$\Sigma\in\R^{d\times d}$ is symmetric and positive definite,
$\phi$ satisfies Assumption \ref{a:scaling-function}, 
$q$ satisfies Assumption \ref{a:proposal} and
$(\eta_n)_{n\ge 2}$ satisfies Assumption \ref{a:step-size}.
Let $W_n\sim q$ be independent and let 
$U_n$ be independent and uniformly distributed in the unit interval 
$[0,1]$. Let $X_1\equiv x_1\in\R^d$ with $\pi(x_1)>0$ and 
$S_1\equiv s_1\in\R$, and for
$n\ge 2$ define recursively
\begin{eqnarray}
    Y_n &=& X_{n-1} + \phi(S_{n-1}) \Sigma^{1/2} W_n 
    \label{eq:asm-proposal} \\
    X_n &=& \begin{cases}
      Y_n,&\text{if $U_n\le \alpha_n$} \\
      X_{n-1},&\text{otherwise}
    \end{cases} \\
    S_n &=& S_{n-1} + \eta_n (\alpha_n - \alpha^*).
    \label{eq:asm-s-adapt}
\end{eqnarray}
\end{definition} %}}}

%%%%%%%%%%%%%%%%%%%%
\begin{definition}[ASWAM]
    \label{def:aswam} %{{{
    Assume the setting of the ASM algorithm in \ref{def:asm}, but 
    instead
    of \eqref{eq:asm-proposal} use
    \begin{eqnarray}
        Y_n &=& X_{n-1} + \phi(S_{n-1}) \Sigma_{n-1}^{1/2} W_n.
    \end{eqnarray}
    The covariance process $(\Sigma_n)_{n\ge 1}$ is determined as follows:
    let $\mu_1\equiv x_1\in\R^d$, suppose
    $\Sigma_1 \in\R^{d\times d}$ is a symmetric and positive definite
    matrix and 
    \begin{eqnarray}
        \hat{\mu}_n &=& (1-\eta_n) \mu_{n-1} + \eta_n X_n 
        \label{eq:aswam-mu} \\
        \hat{\Sigma}_n &=& (1-\eta_n) \Sigma_{n-1}
        + \eta_n (X_n - \mu_{n-1})(X_n - \mu_{n-1})^T \\
        (\mu_n,\Sigma_n) &=& \begin{cases}
          (\hat{\mu},\hat{\Sigma}_n),&\text{if 
            $(\hat{\mu},\hat{\Sigma}_n)\in \spc{S}_\zeta$ and } \\
          (\mu_{n-1},\Sigma_{n-1}),&\text{otherwise,}
        \end{cases}
        \label{eq:aswam-trunc}
    \end{eqnarray}
    where the truncation set is defined as
    $\spc{S}_{\zeta} = \big\{(\mu,\Sigma): \|\mu\|\le \zeta,\,
      \lambda(\Sigma)\subset[\zeta^{-1},\zeta]\big\}$ with
    $\lambda(\Sigma)$ being the set
    of the eigenvalues of $\Sigma$ and $\zeta\in[1,\infty)$ is a constant
    parameter.
\end{definition} %}}}
The step \eqref{eq:aswam-trunc} enforces the stability of the covariance
adaptation process, while the scaling parameter $S_n$ follows 
\eqref{eq:asm-s-adapt}.

Before stating the first ergodicity result, consider the
following condition on the regularity of a collection of sets.
Before that, recall that a $C^1$ domain in $\R^d$ is
a domain whose boundary is locally a graph of a continuously differentiable 
function.
%%%%%%%%%%%%%%%%%%%%
\begin{definition} 
    \label{def:uniformly-continuous-normals} %{{{
    Suppose that $\{A_i\}_{i\in I}$ is a collection of sets
    $A_i\subset\R^d$ each consisting of finitely many 
    disjoint components that are closures of $C^1$ domains.
    Let $n_i(x)$ stand for the outer-pointing normal 
    at $x$ in the boundary $\partial A_i$.
    Then, $\{A_i\}_{i\in I}$ 
    have \emph{uniformly continuous normals} if
    for all $\epsilon>0$ there is a $\delta>0$ such that
    for any $i\in I$ 
    it holds that $\|n_i(x)-n_i(y)\| \le \epsilon$
    for all $x,y\in \partial A_i$ such that 
    $\|x-y\|\le \delta$.
\end{definition} %}}}
Definition \ref{def:uniformly-continuous-normals} 
essentially states that the boundaries $\partial
A_i$ must be regular enough to ensure that 
if one looks at any $\partial A_i$ at a sufficiently small scale, it will look
locally almost like a plane.

%%%%%%%%%%%%%%%%%%%%
\begin{theorem}
    \label{thm:compact-stable} %{{{
    Assume $\pi$ has a compact support $\spc{X}\subset\R^d$ and 
    $\pi$ is continuous, bounded and bounded away from zero on $\spc{X}$. 
    Moreover, assume
    that $\spc{X}$ has a uniformly continuous normal
    (Definition \ref{def:uniformly-continuous-normals}) and 
    $\alpha^*\in\big(0,\frac{1}{2}\big)$.
    Then, for either the ASM or the ASWAM process and for 
    any bounded function $f$, the strong
    law of large numbers holds that is,
    \begin{equation}
        \frac{1}{n}\sum_{k=1}^n f(X_k) \xrightarrow{n\to\infty} \int_{\R^d} f(x)
        \pi(x) \ud x\qquad\text{almost surely.}
        \label{eq:slln}
    \end{equation}
\end{theorem} %}}}
The proof of Theorem \ref{thm:compact-stable} is given in Section
\ref{sec:ergodicity}.

Let us consider next target distributions $\pi$ with unbounded supports,
satisfying the following conditions formulated in \cite{saksman-vihola}.
%%%%%%%%%%%%%%%%%%%%
\begin{assumption} 
    \label{a:super-exp} %{{{
    The density $\pi$ is bounded, bounded away from zero 
    on compact sets, differentiable, and 
\begin{equation}
    \lim_{r\to\infty} \sup_{\norm{x}\ge r}
    \frac{x}{\norm{x}^\rho} \cdot \nabla \log \pi(x) = -\infty
    \label{eq:super-exp}
\end{equation}
for some constant $\rho>1$, where $\|\uarg\|$ stands for the Euclidean norm. 
Moreover, the contour normals satisfy
\begin{equation}
    \lim_{r\to\infty} \sup_{\norm{x}\ge r} 
    \frac{x}{\norm{x}} \cdot \frac{\nabla\pi(x)}{\norm{\nabla\pi(x)}} < 0.
    \label{eq:reg-contour}
\end{equation} 
\end{assumption} %}}}
This assumption is very near to the conditions
introduced by Jarner and Hansen \cite{jarner-hansen} to ensure the geometric
ergodicity of a (non-adaptive) Metropolis algorithm, and
considered by Andrieu and Moulines \cite{andrieu-moulines} in the context of
adaptive MCMC.
In particular, 
\cite{jarner-hansen,andrieu-moulines} assume that $\pi$ fulfils the contour regularity
condition \eqref{eq:reg-contour}. Instead of \eqref{eq:super-exp}, they assume 
a super-exponential decay on $\pi$,
\begin{equation*}
    \lim_{r\to\infty} \sup_{\norm{x}\ge r}
    \frac{x}{\norm{x}} \cdot \nabla \log \pi(x) = -\infty
\end{equation*}
which is only slightly more general than \eqref{eq:super-exp} allowing $\rho=1$.
See \cite{jarner-hansen} for examples and discussion on these conditions.

%%%%%%%%%%%%%%%%%%%%
\begin{theorem} 
    \label{thm:super-exp-ergodicity} %{{{
    Suppose $\alpha^*\in\big(0,\frac{1}{2}\big)$, 
    $\pi$ 
    fulfils Assumption 
    \ref{a:super-exp} and there is a $t_0>0$ such that
    the collection of contour sets $\{x\in\R^d:\pi(x)\ge t\}_{0<t\le t_0}$
    have uniformly continuous normals
    (Definition \ref{def:uniformly-continuous-normals}).
    Assume that there exist constants $c<\infty$ and $p\in(0,1)$ such that 
    $|f(x)| \le c\pi^{-p}(x)$ for all $x\in\R^d$. Then, for the ASM and the
    ASWAM processes, the strong law of
    large numbers \eqref{eq:slln} holds.
\end{theorem} %}}}
The proof of Theorem \ref{thm:super-exp-ergodicity} is given in Section
\ref{sec:ergodicity}.

%%%%%%%%%%%%%%%%%%%%
\begin{remark} %{{{
For many practical target densities satisfying Assumption \ref{a:super-exp}
the tail contours are (essentially) scaled copies of each other, in which
case they have automatically uniformly continuous normals.
This indicates that the conditions of Theorem 
\ref{thm:super-exp-ergodicity} are practically similar to
\citep[Theorem 10]{saksman-vihola} verifying the ergodicity
of the Adaptive Metropolis algorithm.
\end{remark} %}}}

%%%%%%%%%%%%%%%%%%%%
\begin{remark} %{{{
The `safe' values for the desired acceptance rate stipulated by Theorems
\ref{thm:compact-stable} and \ref{thm:super-exp-ergodicity} are
$\alpha^*\in(0,1/2)$. The values $[1/2, 1)$ are excluded due to technical
reasons, in particular due to Proposition \ref{prop:lower} establishing the
lower bound for $\phi(S_n)$. It is
expected that Theorems \ref{thm:compact-stable} and
\ref{thm:super-exp-ergodicity} hold assuming only $\alpha^*\in(0,1)$, but
this cannot be verified with the present approach. The range
$\alpha^*\in(0,1/2)$ is, however, often sufficient in practice, as the 
most commonly used values for a random walk Metropolis algorithms are probably $\alpha^* =
0.234$ and $\alpha^* = 0.44$, and it has been suggested that values
$\alpha^*\in[0.1, 0.4]$ should work well in most cases
\cite{roberts-rosenthal-examples,roberts-rosenthal-scaling,roberts-gelman-gilks-scaling,bedard-spa}.
\end{remark} %}}}

%%%%%%%%%%%%%%%%%%%%
\begin{remark} %{{{
The conditions on the proposal density in Assumption \ref{a:proposal} are
not optimal. The technical tail decay condition on $\hat{q}$ 
is needed in the case of $\pi$ with an unbounded support in
Theorem \ref{thm:super-exp-ergodicity}.
Theorem \ref{thm:compact-stable} considering compactly supported $\pi$ 
can be established for a more general class of proposal 
distributions, but this is not pursued here.
\end{remark} %}}}

%%%%%%%%%%%%%%%%%%%%
\begin{remark} %{{{
Theorems \ref{thm:compact-stable} and \ref{thm:super-exp-ergodicity} ensure
that the trajectories of the ergodic averages converge almost surely but do
not state explicit results on the convergence of the marginal distributions 
of $X_n$. The marginal convergence could be established by
using the technique in
the proof of Proposition 6 of \cite{andrieu-moulines}, modifying it
in the lines of \cite{saksman-vihola}.
\end{remark} %}}}

The rest of the article is organised as follows. Section \ref{sec:notation}
describes a general framework for scale adaptation covering simultaneously
both the ASM and the ASWAM algorithms. Section
\ref{sec:stability} develops stability results for this process.
In particular, Corollary \ref{cor:compact-stability} ensures the stability of
the sequence $\phi(S_n)$ with the assumptions of Theorem
\ref{thm:compact-stable}, and
Proposition \ref{prop:upper-polynomial} controls the growth
of $\phi(S_n)$ when $\pi$ fulfils the conditions of Theorem \ref{thm:super-exp-ergodicity}.
Once the stability results are obtained, Theorems
\ref{thm:compact-stable} and \ref{thm:super-exp-ergodicity} are proved in
Section \ref{sec:ergodicity} using the results in \cite{saksman-vihola}.

%}}}

%%%%%%%%%%%%%%%%%%%%%%%%%%%%%%%%%%%%%%%%%%%%%%%%%%%%%%%%%%%%%%%%%%%%%%%%%%%%%%%
\section{Framework and notation} 
\label{sec:notation} %{{{

Consider a process $(X_n,\Gamma_n)_{n\ge 1}$ evolving in
the measurable space $\spc{X}\times\spc{G}$,
where the support of the target density 
$\spc{X}\defeq\{x\in\R^d:\pi(x)>0\}$
is the space of the `MCMC' chain $(X_n)_{n\ge 1}$, and
the adaptation parameters $(\Gamma_n)_{n\ge 1}=(S_n,\mu_n,\Sigma_n)_{n\ge 1}$ 
evolve in $\spc{G}=\R\times\spc{S}_\zeta$; the scaling parameters
$(S_n)_{n\ge 1}$ are real-valued and the covariance adaptation process
$(\mu_n,\Sigma_n)_{n\ge 1}$ takes values on the space
$\spc{S}_\zeta\subset\R^d\times\mathcal{C}_\zeta$ with
\begin{equation*}
    \mathcal{C}_\zeta\defeq \big\{\Sigma\in\R^{d\times d}: \text{$\Sigma$
is symmetric and $\lambda(\Sigma)\subset[\zeta^{-1},\zeta]$}\big\} 
\end{equation*}
and where
$\lambda(\Sigma)$ stands for the set of eigenvalues of $\Sigma$. 
By this definition, we may define $\spc{S}_\zeta=\{(\mu,\Sigma)\}$
in the case of the ASM whence
$\Sigma_n = \Sigma$ and $\mu_n = \mu$ for all $n\ge 1$
and for the ASWAM, $(\mu_n,\Sigma_n)$ is determined through
\eqref{eq:aswam-mu}--\eqref{eq:aswam-trunc}.
We need the specific form of adaptation of
$(\mu_n,\Sigma_n)$ only in Section \ref{sec:ergodicity}.
For the stability results in Section \ref{sec:stability}
it is sufficient that $\Sigma_n\in\mathcal{C}_\zeta$.

Denote $\F_n \defeq \sigma(W_n,U_n:1\le k\le n)$ so that 
$(\F_n)_{n\ge 1}$ is a filtration and also each $\Gamma_n$ is 
$\F_n$-adapted. With these definitions, we may write
\begin{eqnarray}
    Y_{n+1}\mid \F_n & \sim & q_{\Gamma_n}(X_n, \uarg) 
    \label{eq:prop-rec} \\
    X_{n+1} & = & Y_{n+1} \charfun{\{U_{n+1} \le \alpha_{n+1}\}}
                + X_n \charfun{\{U_{n+1} > \alpha_{n+1}\}}
     \label{eq:x-rec} \\
    S_{n+1} & = & S_n + \eta_{n+1}H(X_n, Y_{n+1})
    \label{eq:s-rec}
\end{eqnarray}
where $\charfun{A}$ stands for the indicator function of a set $A$ and
$H(x,y) \defeq \alpha(x,y) - \alpha^*$ with $\alpha(x,y) \defeq
\min\big\{1,\frac{\pi(y)}{\pi(x)}\big\}$.
Moreover, for $\gamma = (s,\mu,\Sigma)\in\spc{G}$ 
the proposal density is defined as
\begin{equation}
    q_\gamma(z) = 
    q_{(s,\Sigma)}(z) = 
    [\phi(s)]^{-d}\det(\Sigma)^{-1/2}q([\phi(s)]^{-1}\Sigma^{-1/2}z).
    \label{eq:proposal-gamma}
\end{equation}
Note that the form \eqref{eq:s-rec} of adaptation can be 
considered as Robbins-Monro stochastic approximation; see
\cite{andrieu-robert,andrieu-moulines,sa-verifiable} and references
therein.

We will need the notion of 
expected acceptance rate at $x\in\spc{X}$ with parameter
$\gamma\in\spc{G}$ as
\[
    \acc(x,\gamma)
    \defeq \int_{\spc{X}} \alpha(x,y)  q_\gamma(x-y)
    \ud y.
\]
On average, the adaptation rule decreases $S_n$ whenever
$\acc(X_n,\Gamma_n)<\alpha^*$, and vice versa. So, it is
plausible to expect that the algorithm would eventually 
result in $\Gamma_n\to \gamma^*\in\spc{G}$ such that
the overall expected acceptance rate 
$\int_{\spc{X}} \acc(x,\gamma) \pi(x) \ud x = \alpha^*$.
In this paper, however, the convergence of $\Gamma_n$ is not the main concern,
but the stability of it, as it turns out to be crucial for the validity 
of the algorithms considered.

The Metropolis transition kernel with a proposal density $q_\gamma$ is given
as
\begin{equation}
    P_\gamma(x,A) \defeq \charfun{A}(x)\int_{\R^d} [1-\alpha(x,y)]
    q_\gamma(x-y)\ud y
    + \int_A \alpha(x,y) q_\gamma(x-y) \ud y.
    \label{eq:metropolis-kernel}
\end{equation}
Using the kernels $P_\gamma$, one can write \eqref{eq:prop-rec} and 
\eqref{eq:x-rec} as 
$\P(X_{n+1}\in A\mid \F_n) = P_{\Gamma_n}(X_n,A)$.
As usual, integration of a function $f$ with respect to a transition 
kernel is denoted as
\[
    P_\gamma f(x) \defeq \int_{\spc{X}} f(y) P_\gamma(x,\ud y).
\]
Let $V\ge 1$ be a function. The $V$-norm of a function $f$ is defined as
\[
    \| f \|_V \defeq \sup_x \frac{|f(x)|}{V(x)}.
\]
The closed ball in $\R^d$ is written as
$\ball(x,r) \defeq \{y\in\R^d:\|x-y\|\le r\}$, and the distance of a point
$x\in\R^d$ from the set $A\subset\R^d$ is denoted as
$d(x,A)\defeq \inf\{\|x-y\|:y\in A\}$.

%}}}

%%%%%%%%%%%%%%%%%%%%%%%%%%%%%%%%%%%%%%%%%%%%%%%%%%%%%%%%%%%%%%%%%%%%%%%%%%%%%%%
\section{Stability} 
\label{sec:stability} %{{{

This section develops stability results for the general adaptive scaling 
process of Section \ref{sec:notation}. We start with a general stability
theorem based on a martingale argument. This theorem is auxiliary for the
present paper, but may have applications also in other settings.

%%%%%%%%%%%%%%%%%%%%
\begin{theorem} 
    \label{th:stability} %{{{
          Suppose $(\F_n)_{n\ge 1}$ is a filtration and 
          $H_n$ are $\F_n$-adapted random variables satisfying
          $\limsup_{n\to\infty} \eta_{n} H_{n} \le 0$ and
          \begin{equation}
    \sum_{n=2}^\infty \eta_{n}^2 \big(\Econd{H_{n}^2}{\F_{n-1}} 
    -\Econd{H_{n}}{\F_{n-1}}^2\big)
    < \infty.
    \label{eq:stability-martingale-squares}
    \end{equation}
          Let $S_1\equiv s_1\in\R$, and define
          $
              S_{n+1} \defeq S_n + \eta_{n+1} H_{n+1}$ 
              recursively for all
                $n\ge 1$,
          where $\eta_n$ are non-negative constants such that 
          $\sum \eta_n^2 < \infty$.
\begin{enumerate}[(i)]
    \item
      \label{item:adapt-stab1}
If there is a constant $a<\infty$ such that for all $n\ge 1$
\begin{equation*}
    \Econd{ H_{n+1} \charfun{\{S_n\ge a\}}}{\F_n} \le 0,
\end{equation*}
then $\limsup_{n\to\infty} S_n < \infty$ a.s.
\item
  \label{item:adapt-stab2}
  If also $\sum \eta_n =\infty$ and there is a non-decreasing
sequence of $\F_n$-adapted random variables $(A_n)_{n\ge 1}\subset\R$ 
and a constant $b<0$ such that for all $n\ge 1$
\begin{equation*}
    \Econd{ H_{n+1} \charfun{\{S_n\ge A_n\}}}{\F_n} \le 
    b\charfun{\{S_n\ge A_n\}},
\end{equation*}
then $\limsup_{n\to\infty} (S_n - A_n) \le 0$ a.s.
\end{enumerate}
\end{theorem} %}}}
\begin{proof} %{{{
Let $W_{n} \defeq H_n\charfun{\{S_{n-1}\ge a\}}$ for
$n\ge 2$, and define
the martingale $(M_n,\F_n)_{n\ge 1}$ by setting $M_1\defeq 0$, and 
$M_n \defeq \sum_{k=2}^n \ud M_k$ for $n\ge 2$ with the differences $\ud
M_n \defeq \eta_n (W_n - \Econd{W_n}{\F_{n-1}})$. Now,
\[
    \sum_{k=2}^\infty \Econd{\ud M_k^2}{\F_{k-1}} 
    = \sum_{k=2}^\infty \eta_k^2 \big(\Econd{H_{n}^2}{\F_{n-1}} 
    -\Econd{H_{n}}{\F_{n-1}}^2\big)\charfun{\{S_{n-1}\ge a\}}
    < \infty
\]
by assumption. This implies that almost every path
of $M_n$ converges to a finite limit $M_\infty$ \cite[e.g.][Theorem
 2.15]{hall-heyde}.

Let $(\tau_k)_{k\ge 1}$ be the exit times of $S_n$ from $(-\infty,a)$,
defined as $\tau_k \defeq \inf \{n> \tau_{k-1}: S_n\ge a,\,S_{n-1}<a\}$ using the
conventions $\tau_0=0$, $S_0<a$, and $\inf\emptyset = \infty$.
Define also the latest exit from $(-\infty,a)$ until time $n$ by
$\sigma_n \defeq \sup \{\tau_k:k\ge 1,\,\tau_k\le n\}$.
Whenever $S_n\ge a$,
one can write $S_n = S_{\sigma_n} + (M_n -
M_{\sigma_n}) + Z_{\sigma_n,n}$ where 
\[
    \begin{split}
    Z_{m,n} &\defeq 
    \sum_{k={m+1}}^n
    \eta_k \Econd{W_k}{\F_{k-1}} 
    \le 0
    \end{split}
\]
by assumption. In this case,
\begin{equation}
    \begin{split}
    S_n &\le S_{\sigma_n} + (M_n -M_{\sigma_n})  
    \le \max\{S_1, a\} + \eta_{\sigma_n} H_{\sigma_n} + |M_n| +
    |M_{\sigma_n}| \\
    & \le \max\{S_1, a\} + \sup_{k\ge 1} \eta_k H_k
      +2\sup_{k\ge 1} |M_k| \le C
    \end{split}
\label{eq:s-upper-bound}
\end{equation}
where $C$ is a.s.~finite. If $S_n<a$ the claim is trivial
and (\ref{item:adapt-stab1}) holds.

Assume then (\ref{item:adapt-stab2}). If $S_n<A_n$ for all $n$ 
greater than some $N_1(\omega)<\infty$, 
the claim is trivial. Suppose then that $S_n\ge A_n$ infinitely often.
Define $(\tau_k)_{k\ge 1}$ as the exit times of $S_n$ from $(-\infty,A_n)$
as above, $\tau_k \defeq \inf\{n>\tau_{k-1}: S_n\ge A_n,\, S_{n-1}<A_{n-1}\}$
with $\tau_0\equiv 0$ and $S_0<A_0$.
The times $\tau_k$ must be a.s.~finite in this case (and $S_n$ returns to
$(-\infty,A_n)$ infinitely
often), for suppose the contrary: then the last exit times $\sigma_n$
are bounded by some
$\sigma_n\le \sigma<\infty$, and for $n \ge
\sigma$ one may write
\[
    S_n = S_{\sigma} + (M_n - M_{\sigma}) + Z_{\sigma,n}
    \le C_{\sigma} + Z_{\sigma,n}
\]
where $M_n$ and $Z_{n,m}$ are defined as above, but using
the random variables 
$W_{n} \defeq H_n\charfun{\{S_{n-1}\ge A_{n-1}\}}$,
and
the random variable $C_{\sigma}$ is a.s.~finite as in
\eqref{eq:s-upper-bound}. Now,
$Z_{\sigma,n}\to-\infty$ a.s.~as $n\to\infty$, so
$S_n<A_n$ a.s.~for sufficiently large $n$.

Consider then the case $(\tau_k)_{k\ge 1}$ are all finite and $M_n$
converges to a finite $M_\infty$.
Fix an $\epsilon>0$ and
let $N_0 = N_0(\omega,\epsilon)$ be such that for all $n\ge N_0$, 
it holds that
$\eta_{\sigma_n} H_{\sigma_n}\le \epsilon/3$ 
and that $|M_k-M_\infty| \le
\epsilon/3$ a.s.~for all $k\ge \sigma_n$.
The claim follows from the estimate
\[
    \begin{split}
    S_n &\le S_{\sigma_n} + (M_n - M_{\sigma_n})
    =
    S_{\sigma_n-1}
    + \eta_{\sigma_n}H_{\sigma_n}
    + (M_n - M_{\sigma_n}) \\
    &\le A_{\sigma_n} + \epsilon/3 + |M_n - M_\infty| + |M_\infty - M_{\sigma_n}|
    \le A_n + \epsilon
    \end{split}
\]
for all $n\ge N_0$.
\end{proof} %}}}

Hereafter, we shall consider the adaptive scaling process described in 
Section \ref{sec:notation}. One can give simple conditions under which
the result of Theorem \ref{th:stability} applies, since
\[
    \Econd{H(X_n,Y_{n+1})}{\F_n}
    = \acc(X_n,\Gamma_n) - \alpha^*,
\]
so by the boundedness of $H$ 
it is sufficient to find out when
$\acc(x,\gamma)$ is below or above $\alpha^*$.

%%%%%%%%%%%%%%%%%%%%
\begin{lemma} 
    \label{lemma:upper-proposal} %{{{
    Suppose $q$ satisfies Assumption \ref{a:proposal} and $q_{(s,\Sigma)}$
    is defined through \eqref{eq:proposal-gamma}.
    Then,
    there exists a constant $\bar{c}<\infty$ such that
    \begin{equation}
        \sup_{z\in\R^d,\,\Sigma\in\mathcal{C}_\zeta} q_{(s,\Sigma)}(z) 
        \le \bar{c}[\phi(s)]^{-d}
        \qquad
        \text{for all $s\in\R$.}
        \label{eq:proposal-sup}
    \end{equation}
    Moreover, for any $\epsilon>0$ there exist $M<\infty$ 
    such that for all $s\in\R$ and any plane $P\subset\R^d$
    \begin{eqnarray}
        \inf_{\Sigma\in\mathcal{C}_\zeta} 
        \int_{\ball(0,\phi(s)M)} q_{(s,\Sigma)}(z) 
          \ud z &\ge& 1-\epsilon 
        \label{eq:proposal-covering} \\
        \sup_{\Sigma\in\mathcal{C}_\zeta} \int_{\{d(z,P)\le \phi(s) M^{-1}\}}
          q_{(s,\Sigma)}(z) \ud z&\le& \epsilon.
        \label{eq:proposal-plane}
      \end{eqnarray}
\end{lemma} %}}}
The proof of Lemma \ref{lemma:upper-proposal} is straightforward; the details
are  given in Appendix \ref{sec:half-space}.

Let us then record a simple estimate on the expected acceptance rate 
when $\pi$ is compact and $S_n$ is large.
%%%%%%%%%%%%%%%%%%%%
\begin{proposition} 
    \label{prop:upper} %{{{
    Suppose $q$ satisfies Assumption \ref{a:proposal}
    and
    $\pi$ is supported on a compact set $\spc{X}\subset\R^d$
    and $\alpha^* > 0$.
    Then, there is $b<0$ and $a\in\R$ such that
    \begin{equation}
    \Econd{H(X_n,Y_{n+1})}{\F_n}
    \le b\qquad\text{whenever $S_n\ge a$}.
    \label{eq:acc-upper}  
    \end{equation}
\end{proposition} %}}}
\begin{proof} %{{{
    Compute for any $x\in\spc{X}$ and all $\gamma=(s,\mu,\Sigma)\in\spc{G}$
    \[\begin{split}
        \acc(x,\gamma)&=
        \int_{\R^d} \alpha(x,y) q_\gamma(x-y)\ud y
        \le \int_{\ball(x,\mathrm{diam}(\spc{X}))} q_\gamma(z)\ud z \\
        & \le \int_{\ball(x,\mathrm{diam}(\spc{X}))}
        \sup_{\Sigma\in\mathcal{C}_\zeta}q_{(s,\Sigma)}(z)\ud z 
        \le \bar{c}[\phi(s)]^{-d}\int_{\ball(0,\mathrm{diam}(\spc{X}))} \ud z
        \end{split}
    \]
    by \eqref{eq:proposal-sup} in 
    Lemma \ref{lemma:upper-proposal}. We may choose
    $a$ to be sufficiently large so that $\acc(x,\gamma)\le \alpha^*/2$
    whenever $s\ge a$.
    That is, \eqref{eq:acc-upper} holds with $b=-\alpha^*/2<0$, whenever
    $S_n\ge a$.
\end{proof} %}}}

Next, we shall consider the case $S_n$ small, simultaneously for both 
cases where $\pi$ is compactly supported and $\pi$ has a super-exponential
tail.
%%%%%%%%%%%%%%%%%%%%
\begin{proposition} 
    \label{prop:lower} %{{{
    Suppose that 
    there is a $t_0>0$ such that
    $L_{t_0} \defeq \{y\in\R^d:\pi(y)\ge t_0\}$
    is compact and $\pi$ is continuous on $L_{t_0}$.
    Moreover, suppose that 
    the sets in the collection $\{L_t\}_{0<t\le t_0}$ 
    have uniformly continuous normals 
    (Definition \ref{def:uniformly-continuous-normals})
    and $q$ satisfies Assumption \ref{a:proposal}.
    Then, for any $\alpha^* < 1/2$, there are $a\in\R$ and $b>0$ such that
    \begin{equation}
    \Econd{H(X_n,Y_{n+1})}{\F_n}
    \ge b\qquad\text{whenever $S_n\le a$.} 
    \label{eq:acc-lower}\\
    \end{equation}
\end{proposition} %}}}
Before giving the proof of Proposition \ref{prop:lower}, let us outline 
the simple intuition behind it. For all $s$
small enough and for any $\Sigma\in\mathcal{C}_\zeta$, 
the mass of $q_{(s,\Sigma)}$ is essentially concentrated on a
small ball $\ball(0,\epsilon)$. If one looks the target $\pi$
only on $\ball(x,\epsilon)$, there are,  roughly speaking, two alternatives. 
The first one is that
$\pi$ is approximately constant on that small ball and
$\acc(x,\gamma) \approx 1$. The second
alternative is that $\pi$ decreases very rapidly to \emph{one} direction, in which 
case the set $\{y:\pi(y)\ge \pi(x)\}$ looks like a
half-space on the ball $\ball(x,\epsilon)$, and consequently 
$\acc(x,\gamma) \apprge 1/2$.

Before the proof, we shall formulate a 
lemma on this `half-space approximation.'
%%%%%%%%%%%%%%%%%%%%
\begin{lemma} 
    \label{lemma:half-space} %{{{
    Suppose that the sets $\{A_i\}_{i\in I}$ with $A_i\subset\R^d$ 
    have uniformly continuous normals (Definition 
    \ref{def:uniformly-continuous-normals}). Then, for any 
    $\epsilon>0$, there is a $\delta>0$ such that for any $i\in
    I$, any $x\in A_i$ and any $r\in(0,\delta]$,
    there is a half-space $T$ such that
    $\ball(x,r)\cap T \subset \ball(x,r)\cap A_i $, 
    and the distance $d(x,T) \le \epsilon r$.
\end{lemma} %}}}
The claim is geometrically evident. The technical verification is given
in Appendix \ref{sec:half-space}.
%%%%%%%%%%%%%%%%%%%%
\begin{prf-prop-lower} %{{{
    Fix an $\epsilon^*\in(0,1)$ and
    let $M=M(\epsilon^*)$ be the constant from Lemma
    \ref{lemma:upper-proposal} applied with $\epsilon=\epsilon^*$.
    
    By compactness of $L_{t_0}$ and continuity of $\pi$
    one can find $\delta_1>0$ such that for
    all $x,y\in L_{t_0}$ with $\|x-y\|\le \delta_1$, it holds that
    $|\log\pi(x) - \log\pi(y)| \le \epsilon^*$
    so that
    \[
        1 - \alpha(x,y)
        = e^0 - e^{\min\{0,\log\pi(y)-\log\pi(x)\}}
        \le |\log\pi(y)-\log\pi(x)| \le \epsilon^*.
    \]
    Let $\delta_2>0$ be sufficiently small to satisfy 
    Lemma \ref{lemma:half-space} with the choice $\epsilon=M^{-2}$.
        
    Choose a small enough $a\in\R$ so that $2\phi(a) M \le
    \min\{\delta_1,\delta_2\}$.
    Let $s\le a$, denote $r_s \defeq \phi(s)M$, and write
    for any $x\in L_{t_0}$
    \[
        \begin{split}
        \int_{\spc{X}} \alpha(x,y) q_\gamma(x-y) \ud y
        &\ge \int_{\ball(x,r_s)\cap L_{t_0}} \alpha(x,y) q_\gamma(x-y) \ud y
        \\
        &\ge (1-\epsilon^*) 
        \int_{\ball(x,r_s)\cap L_{t_0}} q_\gamma(x-y) \ud y
        \\
        \end{split}
    \]
    since $2r_s\le \delta_1$.
    Denote by $T$ the half-space from Lemma \ref{lemma:half-space}, such that
    $\ball(x,r_s)\cap T \subset \ball(x,r_s)\cap L_{t_0}$ and
    the distance $d(x,T) \le M^{-2} r_s$.
    One obtains
    \[
        \begin{split}
        \int_{\spc{X}} \alpha(x,y) &q_\gamma(x-y) \ud y
        \ge (1-\epsilon^*) 
        \int_{\ball(x,r_s)\cap T} q_\gamma(x-y) \ud y \\
        &\ge (1-\epsilon^*) 
        \int_{\ball(x,r_s)\cap \tilde{T}} q_\gamma(x-y) \ud y 
        - \int_{\{d(y,P)\le M^{-2}r_s\}} q_\gamma(x-y) \ud y\\
        &\ge \frac{1}{2}(1-\epsilon^*)^2 - \epsilon^*
        \end{split}
    \]
    where $\tilde{T}$ is the half-space with the boundary plane $P$
    parallel to the boundary of $T$, and passing through $x$.
    Lemma \ref{lemma:upper-proposal} yields the
    last inequality, specifically 
    \eqref{eq:proposal-covering} with the
    symmetry of $q_\gamma$ and \eqref{eq:proposal-plane}.
    The same estimate clearly holds for any $x\in L_t$ with $t\in(0,{t_0})$.

    To conclude, for any $\alpha^*<1/2$ one can choose
    a sufficiently small
    $\epsilon^*=\epsilon^*(\alpha^*)>0$
    such that
    for all $x\in\spc{X}$ and
    for any $\gamma=(s,\mu,\Sigma)$ with $s\le a$
    \[
            \acc(x,\gamma) = \int_{\spc{X}} \alpha(x,y) q_\gamma(x-y) \ud y
        \ge \frac{1}{2} - \frac{1}{2}\left(\frac{1}{2} - \alpha^* \right).
    \]
    This implies \eqref{eq:acc-lower} with $b = (1/2-\alpha^*)/2>0$.
\end{prf-prop-lower} %}}}
As an easy corollary of the 
propositions above, one establishes the stability of the adaptive scaling
process on the case of compactly supported $\pi$.
%%%%%%%%%%%%%%%%%%%%
\begin{corollary} 
    \label{cor:compact-stability} %{{{
    Suppose $q$ and $(\eta_n)_{n\ge 2}$ satisfy Assumptions \ref{a:proposal}
    and \ref{a:step-size}, respectively,
    $\pi$ has a compact support $\spc{X}\subset\R^d$ and 
    $\pi$ is continuous, bounded and bounded away from zero on $\spc{X}$. 
    Moreover, assume
    that $\spc{X}$ has a uniformly continuous normal
    (Definition \ref{def:uniformly-continuous-normals}).
    Then, for the general adaptive scaling process in Section
    \ref{sec:notation}
    with any $\alpha^*\in\big(0,\frac{1}{2}\big)$
    there exist a.s.~finite random variables $A_1$ and $A_2$ such that
    for all $n\ge 1$
    \begin{equation}
        A_1 \le S_n \le A_2.
        \label{eq:as-stability}
    \end{equation}
\end{corollary} %}}}
\begin{proof} %{{{
    The conditions of Propositions \ref{prop:upper} and \ref{prop:lower} are
    satisfied, so there are constants $-\infty<a_1<a_2<\infty$ and $b<0$
    such that
    \begin{align*}
    \Econd{H(X_n,Y_{n+1})}{\F_n}
    &\le b &&\text{whenever} &S_n&\ge a_2, \\
    \Econd{H(X_n,Y_{n+1})}{\F_n}
    &\ge -b &&\text{whenever} &S_n&\le a_1.
    \end{align*}
    Theorem \ref{th:stability} can be applied to $-S_n$ and $S_n$,
    since by the boundedness of $H$ 
    \eqref{eq:stability-martingale-squares} is implied
    by $\sum \eta_n^2 <\infty$. Theorem \ref{th:stability}
    guarantees that 
    $a_1\le\liminf_{n\to\infty} S_n$ and 
    $\limsup_{n\to\infty} S_n \le a_2$, respectively, from which one obtains a.s.~finite
    $A_1$ and $A_2$ for which \eqref{eq:as-stability} holds.
\end{proof} %}}}

The rest of this section considers targets $\pi$ with an unbounded support.
Under a suitably regular $\pi$, it is shown that the growth of
$S_n$ can be controlled. 
The following estimate for the at most polynomial growth of $\phi(S_n)$ is
crucial for the ergodicity result in Theorem
\ref{thm:super-exp-ergodicity}.
%%%%%%%%%%%%%%%%%%%%
\begin{proposition} 
    \label{prop:upper-polynomial} %{{{
    Suppose 
    $\pi$ 
    fulfils Assumption 
    \ref{a:super-exp} and there is a $t_0>0$ such that
    the collection of contour sets $\{x\in\R^d:\pi(x)\ge t\}_{0<t\le t_0}$
    have uniformly continuous normals
    (Definition \ref{def:uniformly-continuous-normals}).
    Suppose also that $\phi$, $q$ and $(\eta_n)_{n\ge 2}$ satisfy
    Assumptions \ref{a:scaling-function}, \ref{a:proposal} and
    \ref{a:step-size}, respectively.
    Then, for the general adaptive scaling process in Section
    \ref{sec:notation} with $\alpha^*\in\big(0,\frac{1}{2}\big)$, and for any $\beta>0$,
    there exist an a.s.~positive $\Theta_1=\Theta_1(\omega)$ and an a.s.~finite
    $\Theta_2=\Theta_2(\omega,\beta)$ such that for all $n\ge 1$
    \[
    \Theta_1 \le \phi(S_n) \le \Theta_2 n^{\beta}.
    \]
\end{proposition} %}}}
Before the proof, let us consider an estimate of $\acc\big(x,(s,\mu,\Sigma)\big)$ depending on both
$x$ and $s$.
%%%%%%%%%%%%%%%%%%%%
\begin{lemma} 
    \label{lemma:alpha-estim} %{{{
    Assume $q$ satisfies Assumption \ref{a:proposal} and
    $\pi$ satisfies Assumption \ref{a:super-exp}.
    Then, for any $\epsilon>0$, there is a constant $c=c(\epsilon)\ge 1$ such that
    $\acc\big(x,(s,\mu,\Sigma)\big) \le \epsilon$
    for all $\phi(s) \ge c \max\{1,\|x\|\}$.
\end{lemma} %}}}
\begin{proof} %{{{
Let $r_1 \ge 1$ be sufficiently large so that for some $\nu>0$ 
it holds that 
$\frac{x}{\|x\|} \cdot \frac{\nabla \pi(x)}{\|\nabla
 \pi(x) \|} < -\nu$ and 
$\frac{x}{\|x\|^\rho} \cdot \nabla \log \pi (x) < -\nu$ for 
all $\|x\|\ge r_1$.
Increase $r_1$, if necessary, so that for any
$\|x\| \ge r_1$ one can write
$L_{\pi(x)} = \{y:\pi(y)\ge \pi(x)\} = \{ru:u\in S^d,\,0\le r\le g(u)\}$ 
where $S^d \defeq \{u\in\R^d:\|u\|=1\}$ is the unit sphere and
the function $g:S^d\to(0,\infty)$ parameterises the boundary of $L_{\pi(x)}$.
Notice also that the contour normal condition implies the existence of an $M\ge 1$
such that
$L_{\pi(x)} \subset \ball(0,M\|x\|)$ for all $\|x\|\ge r_1$
\cite[see][Lemma 22]{saksman-vihola}.

Write for $\|x\| \ge r_2 \defeq M r_1$ and denoting $T_x \defeq 
\{d(y,L_{\pi(x)})> \|x\|\}$
\[
    \begin{split}
    \acc(x,\gamma) &=
    \int_{\R^d} \alpha(x,y) q_\gamma(x-y) \ud y \\
    &\le \int_{\R^d\setminus T_x} q_\gamma(x-y) \ud y
    + \sup_{y\in\R^d} q_\gamma(x-y)
    \int_{T_x} \alpha(x,y) \ud y.
    \end{split}
\]
The first term can be estimated from above by \eqref{eq:proposal-sup} of
Lemma \ref{lemma:upper-proposal}
\[
    \int_{\ball(0, M\|x\| + \|x\|)} q_\gamma(x-y) \ud y
    \le \bar{c}[\phi(s)]^{-d} \int_{\ball(0, (M+1)\|x\|)}\ud z
    \le c_1 [\phi(s)]^{-d}\|x\|^d
    \le \frac{\epsilon}{2}
\]
whenever $\phi(s)\ge (c_1 2/\epsilon)^{1/d} \|x\|$.

For the integral in the latter term, 
we use polar integration to estimate
\[
    \int_{T_x} \alpha(x,y) \ud y
    \le c_d \sup_{u\in S^d} \int_{r> g(u) + \|x\|}^\infty r^{d-1}
    e^{\log\pi(ru)-\log\pi(g(u)u)} \ud r
\]
where $c_d$ is the surface measure of the sphere $S^d$.
Since $\|x\| \ge r_2$, one has
that $g(u)\ge r_1 \ge 1$, and
from the gradient decay condition, one obtains that 
for $r>g(u)+1$ 
\[
    \begin{split}
    \log\pi(ru) - \log\pi(g(u)u) &= \int_{g(u)}^r \frac{tu}{\|tu\|} \cdot
    \nabla \log\pi(tu) \ud t \le -\nu \int_{g(u)}^r t^{\rho -1} \ud t \\
    &\le -\nu g(u)^{\rho-1} [r-g(u)]
    \end{split}
\]
from which
\begin{multline*}
    \int_{r> g(u) + \|x\|}^\infty r^{d-1}
    e^{\log\pi(ru)-\log\pi(g(u)u)} \ud r \\
    \le \int_0^\infty e^{-\frac{\nu w}{2}} \ud w
    \sup_{r>g(u) + \|x\|}r^{d-1}
    e^{-\frac{\nu}{2} g(u)^{\rho-1} [r-g(u)]}.
\end{multline*}
Consequently,
\[
    \int_{T_x} \alpha(x,y) \ud y  
    \le c_d \frac{2}{\nu}
    \sup_{\tilde{g}\ge 1,\,\tilde{r}>1} \exp \left[(d-1)\log
    (\tilde{g}+\tilde{r}) -
    \frac{\nu}{2} \tilde{g}^{\rho-1}\tilde{r} \right] \le c_2
\]
with a finite constant $c_2$ whenever $\|x\|\ge r_2$.

To sum up, there is a $c_3>0$ such that for any $\|x\| \ge r_2$ and
any $s$ satisfying
\[
    \phi(s) \ge c_3 \max\{1,\|x\|\} \ge 
    \max\bigg\{\left(\frac{2 c_1 }{\epsilon}\right)^{1/d}\|x\|,\,
      \left(\frac{2\bar{c}c_2}{\epsilon}\right)^{1/d}\bigg\},
\]
it holds that $\acc\big(x,(s,\mu,\Sigma)\big) \le \epsilon$. 
For any $\|x\| < r_2$ there is a $r_2\le \|x_0\|\le Mr_2$ 
such that $\pi(x_0) \le \pi(x)$. Consequently, $\alpha(x,y) \le
\alpha(x_0, y)$ for all $y\in \R^d$ and therefore 
\begin{equation*}
    \begin{split}
    \acc(x,\gamma) 
    &\le \int_{\R^d} \alpha(x_0,y)q_\gamma(x-y)\ud y \\
    &\le \int_{\R^d\setminus T_{x_0}} q_\gamma(x-y) \ud y
    + \sup_{y\in\R^d} q_\gamma(x-y)
    \int_{T_{x_0}} \alpha(x_0,y) \ud y.
    \end{split}
\end{equation*}
Repeating the arguments above, there is a finite constant $c_4$ such that
$\acc\big(x,(s,\mu,\Sigma)\big) \le \epsilon$ for all 
$(\mu,\Sigma)\in\spc{S}_\zeta$
and for all $s\in\R$ such that $\phi(s) \ge c_4 \max\{1,\|x\|\}$.
\end{proof} %}}}
Having Lemma \ref{lemma:alpha-estim} and the lower bound from
Proposition \ref{prop:lower}, the proof of Proposition
\ref{prop:upper-polynomial} can be obtained by applying the
growth condition on $\|X_n\|$ established in \cite{saksman-vihola}.
%%%%%%%%%%%%%%%%%%%%
\begin{prf-upper-polynomial} %{{{
Proposition \ref{prop:lower} applied with Theorem \ref{th:stability}
for $-S_n$ 
gives an a.s.~finite $A_1$ such that $A_1 \le S_n$ for all $n\ge 1$.
The random variable $\Theta_1 \defeq \phi(A_1)$ is
a.s.~positive, showing the lower bound.

To check the polynomial growth condition for $\phi(S_n)$, it is first
verified that $\|X_n\|$
grows at most polynomially. 
Fix an $\epsilon>0$ and let $\theta_1=\theta_1(\epsilon)>0$ and
$a_1=a_1(\epsilon)\in \R$
be such that $\theta_1 = \phi(a_1)$, and that $\P(B_1)\ge 1-\epsilon$, 
with $B_1 \defeq \{\Theta_1 \ge \theta_1\}  = \{A_1 \ge a_1\}$.
Let $V(x) \defeq c_\pi \pi^{-1/2}(x)$, where the constant 
$c_\pi \defeq [\sup_x \pi(x)]^{1/2}$ ensures that $V \ge 1$.
Proposition \ref{prop:geom-bound} in Appendix
\ref{sec:metropolis}
shows that the drift inequality 
\begin{equation}
    P_{(s,\Sigma)} V(x)
\le V(x) + b
\label{eq:drift-bound}
\end{equation}
holds for all $\Sigma\in\mathcal{C}_d$ and 
$\phi(s)\ge \theta_1>0$ with some $b=b(\theta_1)<\infty$.
Construct an auxiliary
process $(X_n',\Gamma_n')_{n\ge 1}$ coinciding with $(X_n,\Gamma_n)_{n\ge 1}$
in $B_1$ by setting
$(X_n',\Gamma_n') = (X_{\tau_n},\Gamma_{\tau_n})$ where the stopping times $\tau_n$
are defined as
\[
    \tau_n \defeq \begin{cases}
    n,&\text{if $\phi(S_k)\ge \theta_1$ for all $1\le k\le n$} \\
    \inf \{1\le k\le n-1: \phi(S_{k+1}) < \theta_1 \},
    &\text{otherwise.}
    \end{cases}
\]
Having the inequality \eqref{eq:drift-bound}, set $\beta' = \kappa^{-1}\beta$ 
where the constant $\kappa\ge 1$ is from Assumption \ref{a:scaling-function}
and use
Proposition 7 of \cite{saksman-vihola} 
to obtain the bound $\|X_n'\| \le \Theta_\epsilon n^{\beta'}$ for some
a.s.~finite $\Theta_\epsilon$.
The $\epsilon>0$ was arbitrary, so one can let $\epsilon\to 0$ and obtain
an a.s.~finite $\Theta$ such that $\|X_n\| \le \Theta n^{\beta'}$.
Applying Lemma \ref{lemma:alpha-estim}, 
one obtains that $\acc\big(X_n,(S_n,\Sigma_n)\big) \le \alpha^*/2$ whenever
$\phi(S_n) \ge \Theta' n^{\beta'}$ with $\Theta' \defeq c_1 \max\{1,\Theta\}$.

Fix again an $\epsilon>0$ and let 
$\theta_2=\theta_2(\epsilon)<\infty$ be such that $\P( B_2) \ge 1-\epsilon$
where $B_2\defeq\{\Theta' \le \theta_2\}$. Construct an auxiliary
process $(X_n',S_n')_{n\ge 1}$ coinciding with $(X_n,S_n)_{n\ge 1}$
in $B_2$ by stopping the process if $\phi(S_k)>\theta_2 k^{\beta'}$ as in
the construction above.
Theorem \ref{th:stability} ensures that 
\[
    \limsup_{n\to\infty} [S_n' - \tilde{a}_n ] \le 0
\]
where $\tilde{a}_n$ are defined so that $\phi(\tilde{a}_n) = \theta_2
n^{\beta'}$.
That is, $S_n' \le \tilde{a}_n + E_n$ with $E_n \to 0$ almost surely.
Consider Assumption \ref{a:scaling-function} and
take $N_0$ so large that $E_{n} < h$ for all $n\ge N_0$. Then,
$\phi(x+h) = \phi(x) + h\phi'(x+\bar{h})$ for some $0\le\bar{h}\le h$, and hence
$\phi(x+h) \le c_2 \max\{1,\phi(x)^\kappa\}$. 
For $n\ge N_0$, one has
\[
    \phi(S_n') \le \phi(\tilde{a}_n + E_n) \le c_2
    \max\{1,\phi(\tilde{a}_n)^\kappa\}
    = c_2 \max\{1,\theta_2^\kappa n^{\kappa\beta'}\} \le \theta_2' n^\beta
\]
for some finite $\theta_2'$.
Summing up, there is an a.s.~finite $\Theta_2'$ such that
\[
    \phi(S_n') \le \Theta_2' n^\beta
\]
on $B_2$. Finally, letting $\epsilon \to 0$, one can find 
an a.s.~finite $\Theta_2$ such that $\phi(S_n) \le \Theta_2 n^\beta$.
\end{prf-upper-polynomial} %}}}

%}}}

%%%%%%%%%%%%%%%%%%%%%%%%%%%%%%%%%%%%%%%%%%%%%%%%%%%%%%%%%%%%%%%%%%%%%%%%%%%%%%%
\section{Ergodicity} 
\label{sec:ergodicity} %{{{

Section \ref{sec:stability} established stability or controlled growth for
the adaptive scaling process of Section \ref{sec:notation}.
This section employs these
results to prove strong laws of large numbers in 
Theorems \ref{thm:compact-stable} and
\ref{thm:super-exp-ergodicity} for the ASM and the ASWAM
processes defined in Section \ref{sec:main},
relying on the results introduced in \cite{saksman-vihola}. For this
purpose, consider the following theoretical adaptation
framework introduced in \cite{saksman-vihola} using a
sequence of restriction sets $K_1\subset K_2\subset \cdots
\subset K_n\subset \cdots\subset \spc{G}$.

Assume $(\tilde{X}_n,\tilde{Y}_n,\tilde{\Gamma}_n)_{n\ge 1}$ follow the general 
adaptation framework as described in Section \ref{sec:notation}.
Assume $\tilde{\Gamma}_1\equiv \tilde{\gamma}_1\in K_1$ and instead of
\eqref{eq:s-rec} let $(\tilde{\Gamma}_n)_{n\ge 1}$ follow the
`truncated' recursion
\begin{equation}
    \tilde{\Gamma}_{n+1} = \sigma_{K_{n+1}}\big(\tilde{\Gamma}_n,\,
    \eta_{n+1}\hat{H}(\tilde{X}_{n}, \tilde{Y}_{n+1}) \big)
    \label{eq:s-rec-trunc}
\end{equation}
where the restriction function
$\sigma_{K}:\spc{G}\times\bar{\spc{G}}\to\spc{G}$
is defined as
\begin{equation*}
    \sigma_K(\gamma,\gamma') \defeq \begin{cases}
    \gamma+\gamma',&\text{if $\gamma+\gamma'\in K$} \\
    \gamma,&\text{otherwise},
    \end{cases}
\end{equation*}
$\bar{\spc{G}}\defeq \R\times\R^d\times\R^{d\times d}\supset\spc{G}$
and the function $\hat{H}:\spc{G}\times\spc{X}^2\to\bar{\spc{G}}$ is defined
as
\[
    \hat{H}\big((s,\mu,\Sigma), x,y\big)
    = \begin{bmatrix}
      H(x,y) \\
      x-\mu \\
      (x-\mu)(x-\mu)^T-\Sigma
      \end{bmatrix}
\]
That is, $\sigma_{K_n}$ ensures that $\tilde{\Gamma}_n\in K_n$ for all $n\ge 1$.
Observe that such a `truncated process' can be constructed using an `original
process' $(X_n,\Gamma_n)_{n\ge 1}$ from Section \ref{sec:notation}
and the random variables $(Y_n,U_n)_{n\ge 2}$ following \eqref{eq:x-rec} and
\eqref{eq:s-rec}, so that the two processes coincide in the set
$\cap_{n=1}^\infty \{\Gamma_n\in K_n\}$. 

Before stating the ergodicity result from \cite{saksman-vihola} 
for this truncated chain,
four technical assumptions are listed, which must hold for some constants
$c\ge 1$ and $\beta\ge 0$ and $\iota\in\big(0,\frac{1}{2}\big)$.
\begin{enumerate}[({A}1)]
\item \label{a:invariance} %{{{
  For all measurable $A\subset\spc{X}$, it holds that 
  $\P(\tilde{X}_{n+1}\in A\mid \F_n) = P_{\tilde{\Gamma}_n}(\tilde{X}_n,A)$
  almost surely, and
  for each $\gamma\in\spc{G}$, the transition probability 
  $P_\gamma$ has $\pi$ as the unique invariant
  distribution. %}}}
\item \label{a:uniform-drift-mino} %{{{
  For each $n\ge 1$, the following uniform drift and minorisation 
  conditions hold for all $\gamma\in K_n$, 
  for all $x\in\spc{X}$ 
  and all measurable $A\subset\spc{X}$
\begin{align*}
    P_\gamma V(x) &\le \lambda_n V(x)  + b_n \charfun{C_n}(x)\\
    P_\gamma(x,A) &\ge \delta_n \charfun{C_n}(x)\nu_\gamma(A)
\end{align*}
where $C_n\subset\spc{X}$ is a subset (a minorisation set),
$V:\spc{X}\to[1,\infty)$ is a drift function such that $\sup_{x\in C_n}V(x)
\le b_n$ and $\nu_\gamma$ is a probability measure on $\spc{X}$ concentrated on $C_n$. 
Furthermore, the
constants $\lambda_n\in(0,1)$ and $b_n\in(0,\infty)$
are increasing, $\delta_n\in(0,1]$ is decreasing with respect to $n$
and they are polynomially bounded so that
\begin{equation*}
    \max\{ (1-\lambda_n)^{-1}, \delta_n^{-1}, b_n\}    \le c n^{\beta}.
\end{equation*} %}}}
\item \label{a:kernel-lip} %{{{
  For all $n\ge 1$ and any $r\in(0,1]$, there is $c'=c'(r)\ge 1$ such that
  for all $\gamma$ and $\gamma'$ in $K_n$,
  \begin{equation*}
      \|P_\gamma f - P_{\gamma'} f \|_{V^r} 
      \le c' n^{\beta} \norm{f}_{V^r} | \gamma - \gamma'|
  \end{equation*} 
  with the norm on the space $\bar{\spc{G}}$ defined as 
  $|\gamma|=|(s,\mu,\Sigma)| = |s|+\|\mu\|+\|\Sigma\|$.
  %}}}
\item \label{a:adapt-lip} %{{{
  The inequality
  $|\hat{H}(\gamma,x,y)|\le c n^\beta V^{\iota}(x)$
  holds for all $\gamma\in K_n$ and all 
  $x,y\in\spc{X}$.
  %}}}
\end{enumerate}
%%%%%%%%%%%%%%%%%%%%
\begin{theorem}
    \label{thm:slln-restricted} %{{{
   Assume (A\ref{a:invariance})--(A\ref{a:adapt-lip}) hold 
   and let $f$ be a function
   with $\norm{f}_{V^\tau} < \infty$ for some $\tau\in(0,1-\iota)$. 
   Assume $\beta < \kappa_*^{-1} \min\{
   1/2, 1-\iota-\tau\}$
   and 
   $\sum_{k=1}^\infty k^{\kappa_*\beta-1} \eta_k < \infty$
   where $\kappa_*\ge 1$ is an
   independent constant.
   Then,
   \begin{equation}
       \frac{1}{n} \sum_{k=1}^n f(\tilde{X}_k) \xrightarrow{n\to\infty} 
       \int_{\spc{X}} f(x) \pi(x) \ud x
       \quad\text{almost surely.}
       \label{eq:slln-restricted}
   \end{equation}
\end{theorem} %}}}
\begin{proof} %{{{
    This theorem is a straightforward modification  of
    Theorem 2 in \cite{saksman-vihola}. In particular, the assumption
    (A\ref{a:adapt-lip}) here is only slightly more general 
    than assumption (A4) in
    \cite{saksman-vihola} and the changes required for the proof are obvious.
\end{proof}
%}}}

Now we are ready to give a proof to the 
first main result considering the case of compactly supported $\pi$.
%%%%%%%%%%%%%%%%%%%%
\begin{prf-compact-stable} %{{{
    Corollary \ref{cor:compact-stability} ensures
    that for any $\epsilon>0$, there are 
    $-\infty<a_1^{(\epsilon)}< a_2^{(\epsilon)} < \infty$ such that
    $\P(B^{(\epsilon)})\ge 1-\epsilon$, where
    \[
        B^{(\epsilon)}\defeq \{a_1^{(\epsilon)}\le S_n \le
        a_2^{(\epsilon)}\text{ for all $n\ge 1$}\}.
    \]
    Set $K_n^{(\epsilon)} \defeq K^{(\epsilon)}
    \defeq [a_1^{(\epsilon)}, a_2^{(\epsilon)}]\times\spc{S}_\zeta$ for all 
    $n\ge 1$, and
    construct the truncated process $(\tilde{X}_n^{(\epsilon)},\tilde{\Gamma}_n^{(\epsilon)})$ 
    using these restriction sets in \eqref{eq:s-rec-trunc}. 
    Define $\theta_1^{(\epsilon)} \defeq \phi(a_1^{(\epsilon)})>0$ and
    $\theta_2^{(\epsilon)} \defeq \phi(a_2^{(\epsilon)}) < \infty$.
    
    Let us next verify the above assumptions 
    (A\ref{a:invariance})--(A\ref{a:adapt-lip})
    with some $c\ge 1$, $\beta=0$ and $V\equiv 1$. 
    The assumption 
    (A\ref{a:invariance}) holds by construction of the process and the 
    Metropolis kernel.
    For (A\ref{a:uniform-drift-mino}), take $C_n \defeq \spc{X}$
    for all $n\ge 1$, and notice that $P_\gamma V(x) = 1$ for all $x\in\spc{X}$
    and $\gamma\in\spc{G}$.
    By Assumption \ref{a:proposal} one can estimate
    for all $\gamma\in K^{(\epsilon)}$ and all $x\in\spc{X}$, 
    \[
        \begin{split}
        P_\gamma(x,A) &\ge \int_A \alpha(x,y) q_\gamma(x-y)\ud y \\
        & \ge \left(\inf_{x,y\in\spc{X},\,\gamma\in K^{(\epsilon)}} 
        q_\gamma(x-y)\right)
        \int_A \frac{\pi(y)}{\sup_{z\in\spc{X}} \pi(z)} \ud y \\
        & \ge \theta_2^{-d}\zeta^{-1/2}
        \left(\inf_{|z|\le \mathop{\mathrm{diam}}(\spc{X})} 
            \hat{q}(\|\theta_1^{-1} \zeta^{1/2} z\|)\right)
        c_1 \nu_\gamma(A) \ge \delta \nu_\gamma(A)
        \end{split}
    \]
    with a $\delta>0$,
    where $\nu_\gamma(A)\defeq \nu(A) 
    \defeq c_1^{-1} \int_A \frac{\pi(y)}{\sup_{z\in\spc{X}} \pi(z)} \ud y$
    and $c_1>0$ chosen so that $\nu(\spc{X})=1$.
    Assumption \ref{a:scaling-function} ensures that
    the derivative of $\phi$ is bounded on $[a_1^{(\epsilon)},a_2^{(\epsilon)}]$
    and therefore we have
    \[
    \|\phi(s)\Sigma^{1/2}-\phi(s')\Sigma'^{1/2}\|
    \le \|\Sigma\|\cdot |\phi(s)-\phi(s')|
    + |\phi(s)| \cdot \|\Sigma-\Sigma'\|
    \le c_2 |\gamma-\gamma'|
    \]
    with some finite $c_2=c_2(\epsilon)$
    and 
    Proposition \ref{prop:kernel-lip} in Appendix \ref{sec:metropolis} implies
    (A\ref{a:kernel-lip}). Finally,
    it holds that $|H(\gamma,x,y)| \le c$ for all $\gamma\in K_n$ and
    $x,y\in\spc{X}$, implying 
    (A\ref{a:adapt-lip}).
    
    All (A\ref{a:invariance})--(A\ref{a:adapt-lip}) hold and
    $\sum_{k=1}^\infty k^{-1} \eta_k \le (\sum_{k=1}^\infty
    k^{-2})^{1/2}(\sum_{k=1}^\infty \eta_k^2)^{1/2} < \infty$ 
    by Assumption \ref{a:step-size}, so
    Theorem \ref{thm:slln-restricted} yields a strong law
    of large numbers for the truncated process
    $\tilde{X}_n^{(\epsilon)}$ in case of a bounded function $f$. 
    Since $(\tilde{X}_n^{(\epsilon)})_{n\ge 1}$
    coincides with
    the original process 
    $(X_n)_{n\ge 1}$ in $B^{(\epsilon)}$, the ergodic averages corresponding 
    $X_n(\omega)$ converge to $\int f(x) \pi(x) \ud x$ with almost every~$\omega\in B^{(\epsilon)}$. 
    Since $\epsilon>0$ was
    arbitrary, the strong law of large numbers 
    \eqref{eq:slln} holds almost surely.
\end{prf-compact-stable} %}}}

%%%%%%%%%%%%%%%%%%%%
\begin{remark} %{{{
    Theorem \ref{thm:slln-restricted} (Theorem 2 of
    \cite{saksman-vihola}) is a modification of Proposition 6 in
    \cite{andrieu-moulines}. Having 
    Corollary \ref{cor:compact-stability} ensuring the boundedness of the
    trajectories of $S_n$,
    Theorem \ref{thm:compact-stable} could be
    obtained also using other techniques, in particular, the mixingale
    approach described in \cite{saksman-am,atchade-rosenthal}, or the
    coupling technique of \cite{roberts-rosenthal} (resulting in a weak law
    of large numbers).
    These other techniques do not, however, apply directly to 
    Theorem \ref{thm:super-exp-ergodicity}, since in this case 
    the trajectories of $S_n$ are not necessarily bounded from above,
    but only satisfy the polynomial bound of Proposition
    \ref{prop:upper-polynomial}.
\end{remark} %}}}

%%%%%%%%%%%%%%%%%%%%
\begin{prf-super-exp-ergodicity} %{{{
    Proposition \ref{prop:upper-polynomial} ensures that 
    for any $\beta'>0$ there are
    a.s.~positive $\Theta_1$ and a.s.~finite $\Theta_2$ such that
    \begin{equation}
        \Theta_1 \le \phi(S_n) \le \Theta_2 n^{\beta'}.
        \label{eq:poly-est}
    \end{equation}
    Now, similarly as in the proof of Theorem \ref{thm:compact-stable}, 
    for any $\epsilon>0$, one can find $0<\theta_1^{(\epsilon)}\le
    \theta_2^{(\epsilon)} < \infty$ such that
    \begin{equation}
        \P(\forall n\ge 1: \theta_1^{(\epsilon)} \le \phi(S_n) \le
        \theta_2^{(\epsilon)}n^{\beta'})
        \ge 1-\epsilon
        \label{eq:p-good}
    \end{equation}
    and construct $(\tilde{X}_n^{(\epsilon)},\tilde{S}_n^{(\epsilon)})_{n\ge 1}$ using the 
    restriction sets $K_n^{(\epsilon)} \defeq [a_1^{(\epsilon)},
    a_2^{(n,\epsilon)}]$, where $\phi(a_1^{(\epsilon)}) =
    \theta_1^{(\epsilon)}$ and $\phi(a_2^{(\epsilon,n)}) =
    \theta_2^{(\epsilon)}n^{\beta'}$.
    
    Let $\xi\in(p,1)$ and let 
    $V(x)\defeq c_V \pi^{-\xi}(x)$ with $c_V \defeq \sup_x \pi^{\xi}(x)$.
    Assumption (A\ref{a:invariance}) 
    holds by construction and
    (A\ref{a:adapt-lip}) holds for any given $\iota\in(0,1-\xi)$ 
    as verified in
    the proof of Theorem 10 in \cite{saksman-vihola}, observing that
    $|H(x,y)|\le 1$.
    Proposition \ref{prop:geom-bound} in Appendix
    \ref{sec:metropolis} with the fact $\det(\theta \Sigma) =
    \theta^d \det(\Sigma)$ yields
    (A\ref{a:uniform-drift-mino}) with $\beta = d\beta'$.
    Assumption \ref{a:scaling-function} ensures that $\phi'(s)\le c_1
    \phi^\kappa(s)$ for all $s\in\R$, from which $|\phi(s)-\phi(s')|\le c_1(\theta_2^{(\epsilon)}
    n^{\beta'})^{\kappa}|s-s'|\le c_2 n^{\kappa\beta'} |s-s'|$ for all
    $s,s'\in [a_1^{(\epsilon)},a_2^{(n,\epsilon)}]$. Now,
    Proposition \ref{prop:kernel-lip} in Appendix \ref{sec:metropolis} shows 
    (A\ref{a:kernel-lip}) with $\beta = c_3\beta'$
    as in the proof of Theorem \ref{thm:compact-stable}.
    To conclude, the assumptions (A\ref{a:invariance})--(A\ref{a:adapt-lip}) hold
    with constants $(c,\beta)$, where
    $\beta=\beta(\epsilon,\beta')>0$
    can be selected to be arbitrarily small and $c=c(\epsilon,\beta)<\infty$.
    
    In particular, one can let $\beta<1/2\kappa_*^{-1}$, so that
    $\sum_{k=1}^\infty k^{\kappa_*\beta-1}\eta_k < \infty$ as in the
    proof of Theorem \ref{thm:compact-stable}. Take now
    $\tau=p/\xi\in(0,1)$ 
    so that
    $|f(x)|/V^\tau(x) = c_V^\tau |f(x)|\pi^p(x)$,
    implying that $\|f\|_{V^\tau}<\infty$.
    Theorem \ref{thm:slln-restricted} guarantees that the strong law
    of large numbers holds in the set \eqref{eq:p-good}, and a.s.~by letting
    $\epsilon\to 0$.
\end{prf-super-exp-ergodicity} %}}}

%}}}

%%%%%%%%%%%%%%%%%%%%%%%%%%%%%%%%%%%%%%%%%%%%%%%%%%%%%%%%%%%%%%%%%%%%%%%%%%%%%%%
\section*{Acknowledgements} %{{{

The author thanks Professor Eero Saksman for
comments significantly improving the presentation of the paper. 

%}}}

%%%%%%%%%%%%%%%%%%%%%%%%%%%%%%%%%%%%%%%%%%%%%%%%%%%%%%%%%%%%%%%%%%%%%%%%%%%%%%%
%%%%%%%%%%%%%%%%%%%%%%%%%%%%%%%%%%%%%%%%%%%%%%%%%%%%%%%%%%%%%%%%%%%%%%%%%%%%%%%

\appendix

%%%%%%%%%%%%%%%%%%%%%%%%%%%%%%%%%%%%%%%%%%%%%%%%%%%%%%%%%%%%%%%%%%%%%%%%%%%%%%%
\section{Proofs of geometric lemmas} 
\label{sec:half-space} %{{{

\begin{prf-lemma-proposal} %{{{
Let $\Sigma\in\spc{S}_\zeta$ with $\zeta\in[1,\infty)$, 
that is, the set of eigenvalues satisfy
$\lambda(\Sigma)\subset[\zeta^{-1},\zeta]$. Then
$\zeta^{-d}\le \det(\Sigma)\le \zeta^d$ and
the claim \eqref{eq:proposal-sup} follows by
\[
    \sup_{\Sigma\in\spc{S}_\zeta,\, z\in\R^d}
    q_{(s,\Sigma)}(z)
    \le [\phi(s)]^{-d} \zeta^{d/2}
    \sup_{z\in\R^d} q(z).
\]
Observe then that for any constant $M>0$ one has
\[
        \int_{\ball(0,\phi(s)M)} 
        [\phi(s)]^{-d}\det(\Sigma)^{-1/2}q([\phi(s)]^{-1}\Sigma^{-1/2}z) \ud z
        \ge \int_{\ball(0,\zeta^{-1/2}M)} q(u) \ud u
\]
since $u\in\ball(0,\xi^{-1/2} M)$ implies that $[\phi(s)] \Sigma^{1/2}
u\in\ball(0,\phi(s) M)$.
Clearly $M$ can be chosen sufficiently large so that 
\eqref{eq:proposal-covering} holds.

Let then $P\subset\R^d$ be a plane, and let $z\in\R^d$ such that 
$d(z,P)\le \phi(s) M^{-1}$. Denote by $z^*$ the orthogonal projection of $z$
to $P$, whence $\|z^*-z\|\le \phi(s) M^{-1}$. Denote then 
$u = [\phi(s)]^{-1} \Sigma^{-1/2} z$ and 
$u^* = [\phi(s)]^{-1} \Sigma^{-1/2} z^*$.
We obtain that 
\[
    \|u-u^*\| \le [\phi(s)]^{-1} \zeta^{1/2} \| z - z^*\|
    \le \zeta^{1/2} M^{-1}.
\]
Having this estimate, we can estimate
\[
        \int_{\{d(z,P)\le \phi(s) M^{-1}\}}
        q_{(s,\Sigma)}(z) \ud z
        \le
        \int_{\{d(u,\tilde{P})\le \zeta^{1/2} M^{-1}\}}
          q(u) \ud u
\]
where $\tilde{P} = [\phi(s)]^{-1}\Sigma^{-1/2} P$ is a plane.
To conclude, we may choose $M$ sufficiently large so that
\eqref{eq:proposal-plane} and
\eqref{eq:proposal-covering} hold.
\end{prf-lemma-proposal} %}}}

\begin{prf-half-space} %{{{
    Fix an $\epsilon'>0$.
    By the uniform smoothness of $\{\partial A_i\}_{i\in I}$, one can
    find $\delta>0$ so that $\|n_i(y)- n_i(z)\| \le \epsilon'$
    for all $i\in I$ and $y,z\in\partial A_i$ with $\|y-z\|\le 2\delta$.
    
    Fix an $i\in I$, an $x\in A_i$ and a $r\in[0,\delta]$.
    If $\ball(x,r)\setminus A_i = \emptyset$, one can let $T$ be
    any half-space passing through $x$. Suppose for the rest of the proof
    that $\ball(x,r)\setminus A_i
    \neq \emptyset$ and let $y\in\ball(x,r)\cap \partial A_i$.
    Consider the open cones
    \begin{equation*}
        \begin{split}
        C_-&\defeq \{y+z:n_i(y)\cdot z <
            -\epsilon'\|z\|\} \\
        C_+&\defeq \{y+z:n_i(y)\cdot z >
            \epsilon'\|z\|\} \\
    \end{split}
    \end{equation*}
    illustrated in Figure \ref{fig:half-space-approx}.
    \begin{figure} %{{{
        \psfragscanon
        \psfrag{A}{$A_i$} 
        \psfrag{y}{$y$}
        \psfrag{n}{$n_i(y)$}
        \psfrag{Cm}{$C_-$}
        \psfrag{Cp}{$C_+$}
        \psfrag{B}{$\ball(y,2\delta)$}
        \psfrag{b}{$\ball(x,\delta)$}
        \begin{center}
        \includegraphics{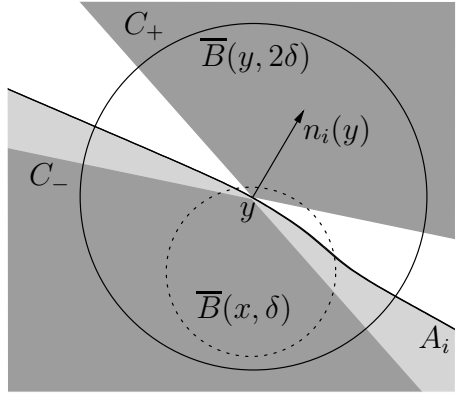}
        \end{center}
        \caption{Illustration of the half-space approximation. The set
          $A_i$ is shown in light grey, and the cones $C_-$ and $C_+$
          in dark grey.} %}}}
        \label{fig:half-space-approx}
    \end{figure}
    We shall verify that
    $\ball(y,2\delta) \cap C_- \subset \ball(y,2\delta) \cap A_i$
    and $\ball(y,2\delta) \cap C_+\subset \ball(y,2\delta) \setminus A_i$.

    Namely, let $u\in  \ball(y,2\delta)\cap C_-$ and write $u = y+z$. 
    Suppose that
    $u\notin A_i$ and define $t_0 \defeq \inf\{t\in[0,1]:y+tz\notin
      A_i\}$.
    Let
    $u_0\defeq y+t_0z$ and notice that
    $u_0\in\ball(y,2\delta)\cap\partial A_i$. Moreover, the line 
    segment $y+tz$
    with $t\in[0,1]$ passes through $\partial A_i$ at $u_0$ and therefore
     $n_i(u_0)\cdot z \ge 0$, since $n_i$ is the outer-pointing normal of
     $A_i$. On the other hand,
     \[
         \begin{split}
         n_i(u_0)\cdot \frac{z}{\|z\|} 
         &= (n_i(u_0) - n_i(y))\cdot \frac{z}{\|z\|} 
         + n_i(y)\cdot \frac{z}{\|z\|} \\
         &< \|n_i(u_0) - n_i(y)\| -\epsilon' < 0,
         \end{split}
     \]
     which is a contradiction, implying
    $C_-\cap \ball(y,2\delta) \subset A_i \cap
    \ball(y,2\delta)$. The case with $C_+$ is verified similarly.
    
    Let us define the half-space 
    $T \defeq \{y - 2\epsilon'r n_i(y) + z: z\cdot n_i(y) < 0\}$. 
    It holds that $\ball(y,2r)\cap T \subset \ball(y,2r)\cap C_-$
    since taking $y + w \in \ball(y,2r)\cap T$ one has
    $n_i(y)\cdot w < -2\epsilon'r \le -\epsilon'\|w\|$.
    On the other hand, 
    $\ball(y,2r)\cap C_- \subset \ball(y,2r) \cap A_i$
    and $\ball(x,r)\subset
    \ball(y,2r)$, so $\ball(x,r)\cap T \subset \ball(x,r)\cap A_i$.
    Clearly, $d(y,T) = 2\epsilon'r$, and
    since $x\notin C_+$ one has $n_i(y)\cdot (x-y) \le
    \epsilon'\|x-y\|\le \epsilon' r$. To conclude, 
    $d(x,T) \le 3\epsilon'r$, and taking $\epsilon' = \epsilon/3$
    yields the claim.
\end{prf-half-space} %}}}

%}}}

%%%%%%%%%%%%%%%%%%%%%%%%%%%%%%%%%%%%%%%%%%%%%%%%%%%%%%%%%%%%%%%%%%%%%%%%%%%%%%%
\section{Simultaneous properties for Metropolis kernels} 
\label{sec:metropolis} %{{{

We shall consider here the following general assumption on the proposal
densities.
%%%%%%%%%%%%%%%%%%%%
\begin{assumption} %{{{
    \label{a:cov-proposal}
    Let $\mathcal{C}_d \subset\R^{d\times d}$ stand for the symmetric and
    positive definite matrices. Suppose $\mathcal{P}\subset\mathcal{C}_d$
    and $\{q_R\}_{R\in\mathcal{P}}$ is a family of probability densities
    defined through
\begin{equation}
    q_R(z) \defeq |\det(R)|^{-1} \hat{q}(\|R^{-1} z\|), 
    \label{eq:cov-prop}
\end{equation}
where $\hat{q}:[0,\infty)\to(0,\infty)$ is a
bounded, decreasing, and differentiable function, satisfying the conditions
in Assumption \ref{a:proposal}.
Moreover, suppose that there is a constant $\kappa>0$ such that
all the eigenvalues of each $R\in\mathcal{P}$ are bounded from below 
by $\kappa$.
\end{assumption}
%}}}

%%%%%%%%%%%%%%%%%%%%
\begin{proposition} 
    \label{prop:geom-bound} %{{{
    Suppose $\pi$ satisfies Assumption \ref{a:super-exp} and
    the family $\{q_R\}_{R\in\mathcal{P}}$ 
    satisfies Assumption \ref{a:cov-proposal} with some $\kappa>0$ 
    and $\beta\in(0,1)$.
    Let $P_R$ be the
    Metropolis transition probability defined in
    \eqref{eq:metropolis-kernel} and using the proposal density $q_R$.
    Then, there exists a compact set $C\subset\R^d$, a probability measure
    $\nu$ on $C$ and a constant $b\in[0,\infty)$ such that 
    for all $R\in\mathcal{P}$, $x\in\R^d$ and
    measurable $A\subset\R^d$,
    \begin{eqnarray}
    P_R V(x) &\le& \lambda_R V(x)  + b \charfun{C}(x) \label{eq:drift-ineq-m} \\
    P_R(x,A) &\ge& \delta_R \charfun{C}(x) \nu(A)  \label{eq:mino-ineq-m}
    \end{eqnarray}
    where
    $V(x)\defeq c_V \pi^{-\beta}(x)\ge 1$ with 
    $c_V \defeq \sup_x \pi^{\beta}(x)$ and the constants
    $\lambda_R,\delta_R\in(0,1)$ satisfy the bound
    \begin{equation*}
        \max\{(1-\lambda_R)^{-1}, \delta_R^{-1}\} \le c
        |\det(R)|^{-1}
    \end{equation*}
    for some constant $c\ge 1$.
\end{proposition} %}}}
\begin{proof} %{{{
    Proposition \ref{prop:geom-bound} is a generalisation of 
    \cite[Proposition 15]{saksman-vihola} considering 
    Gaussian densities $q_{R}$ and the case $\beta=1/2$. 
    We shall describe the changes 
    in the proof of \cite[Proposition
    15]{saksman-vihola} required for the class of proposal distributions in
    Assumption \ref{a:cov-proposal}.
    
    First, observe that with $V(x)=c_V \pi^{-\beta}(x)$ one has
    \[\begin{split}
        1-\frac{P_R V(x)}{V(x)}
        &= \int_{A_x} \left[1-\left(\frac{\pi(x)}{\pi(y)}\right)^{\beta}\right]
        q_R(y-x)\ud y \\
        &\phantom{=}- \int_{R_x} \left(\frac{\pi(y)}{\pi(x)}\right)^{1-\beta}
        \left[1-\left(\frac{\pi(y)}{\pi(x)}\right)^\beta\right] 
            q_R(y-x)\ud y.
    \end{split}\]
    The $1/4$ in the estimate (37) of \cite{saksman-vihola}
    is replaced with $c_*=\sup_{u\in[0,1]} u^{1-\beta}(1-u^\beta)\in(0,1)$.
    One can easily make
    $1-\big(\pi(x)/\pi(y)\big)^\beta>c^*$ for all $y\in\tilde{A}_x$,
    where $c^*$ is any chosen value in $(c_*,1)$.
    
    For a non-negative function
    $f$,  one can write 
    by Fubini's theorem
    \[
        \begin{split}
        \int_{\R^d} f(z+x) q_R(z) \ud z &= 
        |\det(R)|^{-1} \int_0^{\hat{q}(0)} 
          \int_{\{\hat{q}(\|[R^{-1} z\|) \ge t\}} f(z+x)\ud z 
        \ud t
        \\
        &= -
         |\det(R)|^{-1} \int_0^\infty 
          \int_{E_u} f(y) \ud y \hat{q}'(u) \ud u
        \end{split}
    \]
    where the substitution $t = \hat{q}(u)$ was used,
    and $E_u \defeq \{x+z:\|R^{-1} z\|\le u\}$.
    One has
    $\|R^{-1} z\| \le \kappa^{-1} \|z\|$, and
    thus $E_u \supset \ball(x,u\kappa)$.
    
    The conditions in Assumption \ref{a:proposal} for
    the derivative $\hat{q}'$ correspond to the estimate obtained in 
    \cite[Lemma 14]{saksman-vihola}
    for a Gaussian family, that is, $\hat{q} = e^{-x^2/2}$ and the case 
    $\xi=1/2$. In the present case, the choice $\xi=c_*/c^*$ 
    is used.
    These facts are enough to complete the proof of
    \cite[Proposition
    15]{saksman-vihola} to yield the claim.
\end{proof}
%}}}

%%%%%%%%%%%%%%%%%%%%
\begin{proposition}
    \label{prop:kernel-lip} %{{{
    Suppose the family $\{q_R\}_{R\in\mathcal{P}}$
    satisfies Assumption \ref{a:cov-proposal} with some $\kappa>0$.
    Suppose, in addition, that either
\begin{enumerate}[(i)]
\item \label{item:v1} $V \equiv 1$ or
\item \label{item:v2} $\pi$ satisfies Assumption \ref{a:super-exp} and
  $\beta\in(0,1)$, 
$V(x)\defeq c_V \pi^{-\beta}(x)\ge 1$ with 
$c_V \defeq \sup_x \pi^{\beta}(x)$.
\end{enumerate}
Then, there is a constant $c>0$ such that
for the Metropolis transition probability $P_R$ given
in \eqref{eq:metropolis-kernel}, it holds that
\begin{equation}
    \norm{P_R f - P_{R'} f }_{V^r} \le c
    \max\{\|R\|,\|R'\|\}^{d+1} \|f\|_{V^r} \|R-R'\|
    \label{eq:kernel-lip}
\end{equation}
for all $R,R'\in\mathcal{P}$ and $r\in[0,1]$.
The matrix norm above is the Frobenius norm defined as $\|R\| \defeq
\sqrt{\tr(R^T R)}$.
\end{proposition} %}}}
\begin{proof} %{{{
Consider first (\ref{item:v1}). From the definition of
the Metropolis kernel \eqref{eq:metropolis-kernel}, one obtains
\[
    \sup_x | P_R f(x) - P_{R'} f(x)|
    \le 2 \sup_x |f(x)| \int_{\spc{X}} |q_R(x)-q_{R'}(x)| \ud x.
\]
For (\ref{item:v2}),
Proposition 12 of \cite{andrieu-moulines} shows that
for any $r\in[0,1]$
it holds that
\[
    \norm{P_R f - P_{R'} f }_{V^r} 
    \le 2 \norm{f}_{V^r} \int_{\R^d}
    |q_R(x)-q_{R'}(x)| \ud x
\]
so it is sufficient to consider only the total variation of the proposal
distributions.

As in \cite{saksman-am} and \cite{andrieu-moulines}, one can write
\[
    \int_{\spc{X}} |q_R(x)-q_{R'}(x)| \ud x
    = \int_{\spc{X}} \left|\int_0^1 \frac{\ud}{\ud t}q_{R_t}(x) \ud t \right| \ud x
\]
where $R_t \defeq R' + t(R-R')$.
Let us compute
\[
    \frac{\ud}{\ud t}q_{R_t}(x)
    = |\det(R_t)|^{-1}\left(
    - \tr\big(R_t^{-1}(R-R')\big)q_{R_t}(x)
      + \hat{q}'(\|R_t^{-1} x\|) \frac{\ud}{\ud
        t}\|R_t^{-1} x\|\right)
\]
and
\[
    \frac{\ud}{\ud
        t}\|R_t^{-1} x\|
    = - \left(\frac{R_t^{-1} x}{\|R_t^{-1} x\|}\right)^T
    R_t^{-1}(R-R')R_t^{-1} x.
\]
Since $R-R'$ and $R_t^{-1}$ are symmetric and 
$R_t^{-1}$ positive definite, it holds 
that $|\tr\big(R_t^{-1}(R-R')\big)|\le \tr(R_t^{-1}) \max_{1\le i\le d}
|\lambda_i|
\le \tr(R_t^{-1}) \|R-R'\|$ where $\lambda_i$ are the eigenvalues of
$R-R'$ \cite[see, e.g, ][]{wang-matrix-trace-ineq}.
Since the Frobenius norm is sub-multiplicative,
\[
    \begin{split}
    \int_{\spc{X}} &|q_R(x)-q_{R'}(x)| \ud x \\
    &\le \sup_{t\in[0,1]} |\det(R_t)|^{-1}\left(\tr(R_t^{-1}) 
    +  \|R_t^{-1}\|^2
    \int_{\spc{X}} \|x\| \left|\hat{q}'(\|R_t^{-1}x\|)\right|\ud x  
    \right)\|R-R'\| \\
    &\le  \kappa^{-d}\left( d\kappa^{-1}
    + d\kappa^{-2} c_d \sup_{\|u\|=1,\,t\in[0,1]} \int_0^\infty r^{d} |
    \hat{q}'(r\|R_t^{-1}u\|)| \ud r 
    \right)\|R-R'\| 
    \end{split}
\]
by polar integration. Denote $\lambda = \lambda(u,t) \defeq \|R_t^{-1} u\|$, and observe that
since $\hat{q}$ is decreasing, integration by parts yields
\[
    \begin{split}
    \int_0^M r^{d} |
    \hat{q}'(\lambda r)| \ud r 
    &= \frac{d}{\lambda} \int_0^M r^{d-1} \hat{q}(\lambda r) \ud r
    - M^d \frac{\hat{q}(\lambda M)}{\lambda}  \\
    &\le\frac{d}{\lambda^{d+1}} \int_0^\infty u^{d-1} \hat{q}(u) \ud u
    = \frac{dc_{\hat{q}}}{\lambda^{d+1}}
    \end{split}
\]
for all $M>0$.
Since $\lambda^{-1}$ is smaller, for any $\|u\|=1$ and $t\in[0,1]$, than the 
maximum eigenvalue of $R$ and $R'$,
which is smaller than $\max\{\|R\|,\|R'\|\}$, we obtain
\[
    \int_{\R^d} |q_R(x)-q_{R'}(x)| \ud x 
    \le c_1 \max\{\|R\|,\|R'\|\}^{d+1} \|R-R'\|
\]
concluding the proof with $c=2c_1$.
\end{proof} %}}}

%%%%%%%%%%%%%%%%%%%%
\begin{proposition}
\label{prop:proposal-examples} %{{{
Suppose the proposal density $q$ is 
given as $q(z) = c\tilde{q}(\| z\|)$ where $c>0$ is a constant and 
\begin{enumerate}[(i)]
    \item \label{item:q0-exp2} $\tilde{q}(x) = e^{-x^2/2}$, or
    \item \label{item:q0-poly} $\tilde{q}(x) = (1+x^2)^{-d/2-p}$ for some
      $p>0$.
\end{enumerate}
That is, $q$ is a (multivariate) Gaussian or Student distribution, respectively.
Then, $q$ satisfies Assumption \ref{a:proposal}.
\end{proposition} %}}}
\begin{proof} %{{{
It is sufficient to verify that the derivative of 
$\tilde{q}$ satisfies the conditions in
Assumption \ref{a:proposal}.
Fix $\xi\in(0,1)$ and assume $\epsilon>0$. 
Consider first (\ref{item:q0-exp2}), in which case
\[\begin{split}
    \xi\tilde{q}'(x) - \tilde{q}'(x+\epsilon) 
    &= (x+\epsilon) e^{-(x+\epsilon)^2/2} - \xi x e^{-x^2/2} \\
    &\ge xe^{-x^2/2} \left[e^{-\epsilon x - \epsilon^2/2} - \xi\right]
    >0
\end{split}\]
if and only if $x<x_\epsilon\defeq -\frac{\epsilon}{2} -\frac{\log \xi}{\epsilon}$.
Let $\epsilon_*\in(0,1)$ be small enough so that $x_\epsilon>0$ for all
$\epsilon\in(0,\epsilon_*]$, from which one
obtains $c_1>0$ and $0\le a < b <\infty$ such that 
$\xi\tilde{q}'(x) - \tilde{q}'(x+\epsilon) \ge c_1$ for all $x\in[a,b]$
and all $\epsilon\in[0,\epsilon_*]$.
Moreover, for all $\epsilon\in(0,\epsilon_*)$
\[\begin{split}
    \int_0^\infty \!\!\min\{0,\xi\tilde{q}'(x) - \tilde{q}'(x+\epsilon)\} \ud x
    &\ge \int_{x_\epsilon}^\infty 
    xe^{-x^2/2} \left[e^{-\epsilon x - \epsilon^2/2} - \xi\right] \ud
    x 
    \ge - \xi e^{-x_\epsilon^2/2} \\
    &= - \xi e^{-\epsilon^2/8 -\log(\xi)/2}
    e^{-(\log \xi)^2\epsilon^{-2}/2}
    \ge -c_2 e^{-c_3 \epsilon^{-1}}
\end{split}\]
with $c_2=\xi e^{-\log(\xi)/2}$ and $c_3=(\log\xi)^2/2$.

Assume then (\ref{item:q0-poly}).
By the mean value theorem, denoting $c \defeq d+2p$ and $\alpha \defeq
d/2+p+1$,
one can write for some $\epsilon'\in[0,\epsilon]$
\[
    \begin{split}
    \xi\tilde{q}'(x) - \tilde{q}'(x+\epsilon) 
    &\ge c x \left(\frac{1}{(1+(x+\epsilon)^2)^{\alpha}}
    - \frac{\xi}{(1+x^2)^{\alpha}}\right) \\
    &= c x \left(\frac{1-\xi}{(1+(x+\epsilon)^2)^{\alpha}}
    - \frac{2\xi\alpha\epsilon(x+\epsilon')}{(1+(x+\epsilon')^2)^{\alpha+1}}\right) \\
    &\ge \frac{c(1-\xi) x}{(1+(x+\epsilon)^2)^{\alpha}} \left( 1
    - \frac{2\xi\alpha\epsilon}{1-\xi}\left(\frac{1+(x+\epsilon)^2}{1+(x+\epsilon')^2}\right)^{\alpha}\right) 
    > 0
    \end{split}
\]
for all $x>0$, whenever $\epsilon>0$ is sufficiently small. The claim
follows easily.
\end{proof} %}}}

%}}}

\end{document}